\documentclass[letterpaper,11pt]{article}
\usepackage[backend=biber,style=alphabetic,citestyle=alphabetic,giveninits=true,maxnames=5,maxbibnames=9, maxcitenames=6,mincitenames=4,maxcitenames=5,doi=false,isbn=false,url=false]{biblatex} 
\AtEveryBibitem{\clearlist{language}}
\usepackage{tikz-cd}
\AtEveryBibitem{\clearlist{language}}
\usepackage{tikz-cd}
\renewbibmacro*{in:}{%
  \ifentrytype{article}
    {}
    {\bibstring{in}%
     \printunit{\intitlepunct}}}
\usepackage{graphicx} 
\usepackage{palatino,mathrsfs}
\usepackage[mathcal]{euler}
\usepackage{xcolor}
\usepackage{epstopdf,epsfig}
\usepackage{pdfpages}

\usepackage[colorlinks=true,linkcolor=black,hypertexnames=false]{hyperref}
\usepackage{comment} 
\usepackage{amsmath,amsthm,amsfonts,amssymb,latexsym,amscd,enumerate,url,hyperref}
\usepackage{amssymb}

\usepackage{mathtools}
\usepackage{graphicx}
\usepackage[all,knot]{xy}

\xyoption{arc}
\usepackage{multirow}
\usepackage{caption}
\usepackage{soul}

\newtheorem*{theorem*}{Theorem}
  
 \newtheorem*{corollary*}{Corollary}

\newtheorem*{lemma*}{Lemma}

\newtheorem*{proposition*}{Proposition}

\newtheorem*{claim*}{Claim}
\theoremstyle{definition}

\newtheorem*{definition*}{Definition}
 
\newtheorem*{example*}{Example}

\newtheorem*{question*}{Question}
  
\newtheorem*{problem*}{Problem}
  
\newtheorem*{remark*}{Remark}

\def\R{\mathbb R}
\def\C{\mathbb C}

\def\fg{\mathfrak{g}}
\def\fk{\mathfrak{k}}

\def\fa{\mathfrak{a}}

\def\fl{\mathfrak{l}}
\def\fp{\mathfrak{p}}
\def\fq{\mathfrak{q}}%
\def\g{\boldsymbol{\mathfrak{g}}}
\def\k{\boldsymbol{\mathfrak{k}}}
\def\p{\boldsymbol{\mathfrak{p}}}

\def\l{\boldsymbol{\mathfrak{l}}}

\title{Dirac operators for Cartan motion groups}

\author{Spyridon Afentoulidis-Almpanis \& Eyal Subag
}

\newcommand{\Addresses}{{
		\bigskip
		\footnotesize
		
		Spyridon~Afentoulidis-Almpanis \par\nopagebreak\textsc{Dept. of Mathematics, Bar-Ilan University, Ramat-Gan, 5290002 Israel}\par\nopagebreak
		\textit{E-mail address}: \texttt{spyridon.almpanis@biu.ac.il}
		
		\medskip
		
Eyal~Subag \par\nopagebreak\textsc{Dept. of Mathematics, Bar-Ilan University, Ramat-Gan, 5290002 Israel}\par\nopagebreak
\textit{E-mail address}: \texttt{eyal.subag@biu.ac.il}
		
		\medskip

}}

\date{\today}

\begin{document}

\maketitle

 \begin{abstract} We introduce and study the algebraic Dirac operator for the Cartan motion group associated with a real reductive Lie group. We establish its fundamental properties, including a square formula involving a natural Casimir-type element, define the corresponding notion of Dirac cohomology, and prove a strong form of Vogan's conjecture: a module in the Dirac series is completely determined by its Dirac cohomology. We explicitly determine the Dirac series of a Cartan motion group.

We explain how the Dirac theory of a Cartan motion group can be obtained as the zero fiber of the recently introduced Dirac theory for the algebraic deformation family of the corresponding real reductive group.
	\end{abstract}


\tableofcontents

\section{Introduction}

 \subsection{Aim}
 The algebraic Dirac operator associated with a real reductive group $G(\R)$ is a
 useful tool in the study of representations of $G(\R)$; among other
 things, it detects the infinitesimal character of representations
 with nonzero Dirac cohomology and provides an effective criterion
 for disproving unitarity of certain classes of representations.

 Attached to $G(\R)$ there is a semidirect product Lie group
 $G(\R)_0$, the so-called Cartan motion group of $G(\R)$. In general,
 $G(\R)_0$ is non-reductive; in particular, the classical construction of the
 algebraic Dirac operator does not apply to it verbatim. There is
 also a one-parameter family of groups, the deformation family,
 interpolating between $G(\R)$ and $G(\R)_0$.

 The purpose of this paper is to introduce and study the algebraic
 Dirac operator for Cartan motion groups. We determine the Dirac
 series of a Cartan motion group completely, and we prove a strong
 form of Vogan's conjecture in this setting: the Dirac cohomology of
 a module in the Dirac series determines the module itself. In
 addition, we discuss some relations between the recently introduced
 Dirac theory of the deformation family of $G(\R)$ and the Dirac
 theory of the corresponding Cartan motion group.

 In the next subsections we describe the motivation for this paper
 and our main results.
 
\subsection{Motivation and context }\label{MotivationAndContext}

Let $G(\R)$ be a noncompact connected real reductive group with a Cartan involution $\theta$ such that $K(\R):=G(\R)^{\theta}$ is a maximal compact subgroup. On the level of Lie algebras, $\theta$ gives rise to the Cartan decomposition
\[\fg(\R):=\operatorname{Lie}(G(\R))=\fk(\R)\oplus \fp(\R), \]
where $\fk(\R)=\operatorname{Lie}(K(\R))$ is the $(+1)$-eigenspace of $\theta$ and $\fp(\R)$ is the $(-1)$-eigenspace. The group $K(\R)$ naturally acts on the vector space $\fp(\R)$ by restriction of the adjoint action. This gives rise to a semidirect product group, the \textit{Cartan motion group} of $G(\R)$, given by
\[G(\R)_0:=\fp(\R)\rtimes K(\R),\]
in which the normal subgroup $\fp(\R)$ is an abelian vector group.

The terminology is rooted in geometry. The homogeneous space $X=G(\R)/K(\R)$ is a Riemannian symmetric space of noncompact type, and its tangent space at the base point $eK(\R)$ is naturally identified with $\fp(\R)$. The group $G(\R)_0$ acts on the flat space $\fp(\R)$ by \textit{motions}: the subgroup $\fp(\R)$ acts by translations and $K(\R)$ acts by linear isometries. 
Harmonic analysis on the flat symmetric space $\fp(\R)\cong G(\R)_0/K(\R)$ was studied by Helgason, who developed the tangent space analysis relating functions on $X$ to functions on $\fp(\R)$ \cite{Helgason1980}; for the theory of spherical functions on Cartan motion groups see the work of Rader \cite{MR965746}.

Being a semidirect product with an abelian normal subgroup, the group $G(\R)_0$ is usually non-reductive, but its representation theory is well understood. Its unitary dual admits a complete description via the Mackey machine \cite{zbMATH03536358}. The admissible dual of $G(\R)_0$ was determined by Champetier and Delorme \cite{delormeChampetier} and by Rader \cite{MR965746} when $G(\R)$ is semisimple with a finite center, and recently, in the general real reductive case, by Gaudillot-Estrada \cite{GaudillotEstrada2026CovariantRepresentations}.

Although $G(\R)_0$ is a very different group from $G(\R)$, their representation theories are deeply related. Motivated by the notion of contraction of Lie groups from mathematical physics, see e.g. \cite{MR55352,MR779059,MR2942592,MR3077834}, Mackey noticed that a large part of the unitary dual of a noncompact semisimple Lie group closely resembles a large part of the unitary dual of its Cartan motion group, and he suggested that a precise correspondence between the two should exist \cite{MR409726}.

The subject was taken up again decades later by Higson who, for complex semisimple groups, sharpened Mackey's analogy into a bijection, now known as the \textit{Mackey--Higson bijection}: in \cite{MR2391803} for the tempered duals of $G(\R)$ and $G(\R)_0$, and in \cite{MR2815133} for the admissible duals. He also used the bijection to reprove the Connes--Kasparov conjecture for complex semisimple groups.
 Later on, Afgoustidis generalized Higson's results to real reductive Lie groups \cite{MR4079418,MR4400734,MR3989145}; see also \cite{clare1,clare2} for related recent work on the Connes--Kasparov isomorphism. For an approach to the Mackey analogy through deformations of $\mathcal{D}$-modules, see \cite{MR4542720}. Let us also mention some further recent developments in this circle of ideas. Clare, Higson, and Rom\'an constructed an embedding of the group $C^*$-algebra of a Cartan motion group into the reduced $C^*$-algebra of $G(\R)$, and used it to give a $C^*$-algebraic characterization of the Mackey--Higson bijection \cite{ClareHigsonRoman2025}; see \cite{HigsonRoman2020} for an earlier treatment of the complex case. Afgoustidis and Clare analyzed this embedding as a stratified equivalence and, along the way, described the unitary dual of a Cartan motion group as a spectral extended quotient \cite[Thm.~1.16]{AfgoustidisClare2026}. Bradd, Higson, and Yuncken proved that the Connes--Kasparov isomorphism is equivalent to a $K$-theoretic form of Vogan's theorem \cite{MR2401817} on minimal $K$-types of tempered representations with real infinitesimal character \cite{BraddHigsonYuncken2024}.

Many of the above works exploit, in one form or another, a family of Lie groups $\{G_t\}_{t \in [0,1]}$ satisfying
\[
G_t \cong
\begin{cases}
G(\R) & t \neq 0, \\
G(\R)_0 & t = 0.
\end{cases}
\]
Dirac operators already appeared in this circle of ideas. Recall that, by the work of Parthasarathy and of Atiyah and Schmid \cite{Par,MR463358}, the discrete series representations of $G(\R)$ can be realized via geometric Dirac operators on $G(\R)/K(\R)$. Moreover, some principal series representations can be embedded in kernels of suitable Dirac operators \cite{MZ1,MZ1b}. 

In the case of an equal-rank group $G(\R)$, Afgoustidis showed that the unitary irreducible representations of $G(\R)_0$ that correspond to discrete series representations of $G(\R)$ via the Mackey--Higson bijection can be constructed by letting suitably twisted geometric Dirac operators vary along $\{G_t\}_{t \in [0,1]}$ \cite{MR4079418}.

The above-mentioned family of Lie groups is closely related to an algebraic family $(\g_d,K)$ of Harish-Chandra pairs over $\mathbb{A}^1_{\C}$, the \textit{deformation family} associated with $G(\R)$. Algebraic families of Harish-Chandra pairs, together with their modules, were first studied systematically by Bernstein, Higson, and the second author \cite{eyalbern,MR4123111}, who developed them into an algebraic framework for the study of contractions of Lie groups and of their representations. For $t\neq 0$ the fiber of $(\g_d,K)$ at $t$ is the underlying Harish-Chandra pair of $G(\R)$, and for $t=0$ it is the underlying Harish-Chandra pair of $G(\R)_0$. In \cite{eyallietheory}, the second author employed algebraic families of modules for $(\g_d,K)$ to describe the Mackey--Higson bijection algebraically in the case of $G(\R) = SL_2(\R)$, and proposed a conjectural extension of this algebraic description to all real reductive groups.

On the Dirac-theoretic side, the algebraic counterpart of the geometric Dirac operator of $G(\R)$ is the algebraic Dirac operator, an element of $\mathcal{U}(\fg)\otimes_{\C} \operatorname{Cl}(\fp)$ introduced by Vogan \cite{vogantalks}. Vogan attached to any Harish-Chandra module its Dirac cohomology and conjectured that, when nonzero, the Dirac cohomology determines the infinitesimal character of the module. The conjecture was proved by Huang and Pand\v{z}i\'{c} \cite{huangpandzic}, and Dirac cohomology has since become a standard tool in the representation theory of real reductive groups; see \cite{pandzic} and the survey \cite{Renard2014DiracSurvey}. Recently, the authors constructed Dirac operators and Dirac cohomology for the deformation family $(\g_d,K)$, as well as for other related families, and established in that setting an analogue of Vogan's conjecture \cite{afentoulidisalmpanis2025diracoperatorsalgebraicfamilies}.

As mentioned above, the zero fiber of the deformation family $(\g_d,K)$ is the underlying Harish-Chandra pair of $G(\R)_0$. Against this background, the following questions present themselves.

\begin{itemize}
    \item Is there an intrinsic algebraic Dirac operator for the Cartan motion group $G(\R)_0$? Note that since $\fg_0$, the complexified Lie algebra of $G(\R)_0$, is non-reductive, it does not, in general, admit a nondegenerate invariant symmetric bilinear form, and the classical construction does not directly apply.
    \item Assuming such an operator exists, does its square admit a formula in terms of a Casimir-type element, and does it lead to a good notion of Dirac cohomology?
    \item Which irreducible unitary representations of $G(\R)_0$ have nonzero Dirac cohomology? That is, what is the Dirac series of $G(\R)_0$, and does an analogue of Vogan's conjecture hold for it?
    \item Can the Dirac theory of $G(\R)_0$ be obtained from the Dirac theory of the deformation family by specialization at zero?
\end{itemize}

In the present paper we give complete answers to all four questions.

\subsection{Main results and structure of the paper}\label{MainResults}

Section \ref{DiracTheory} develops the algebraic Dirac theory of a Cartan motion group intrinsically, that is, entirely within $G(\R)_0$. Writing $\fg_0=\fk\oplus \fp_0$ for the complexified Lie algebra of $G(\R)_0$, we observe that even though $\fg_0$, in general, admits no nondegenerate invariant symmetric bilinear form, the $K$-module $\fp_0$ carries a natural $(\fg_0,K)$-invariant nondegenerate symmetric form $\beta_0$, obtained by restricting an extension $\beta$ of the Killing form of the derived algebra of $\fg$. This allows us to associate with $G(\R)_0$ the same Clifford algebra $\operatorname{Cl}(\fp_0,\beta_0)$ that is associated with $G(\R)$, and to define, in a canonical coordinate-free manner, the algebraic Dirac operator
\[D_0\in A_0:=\mathcal{U}(\fg_0)\otimes_{\C}\operatorname{Cl}(\fp_0,\beta_0)\]
of $G(\R)_0$. We also introduce a canonical Casimir-type element $\Omega_0$, which lies in the $K$-invariant part of the center of $\mathcal{U}(\fg_0)$, and prove the square formula
\[2D_0^2=\Omega_0\otimes 1_{\operatorname{Cl}(\fp_0,\beta_0)}.\]
With these constructions in hand, the Dirac cohomology $H_{D_0}(V)$ of a $(\fg_0,K)$-module $V$ is defined, in analogy with the reductive case, as the $\widetilde{K}$-module
\[
H_{D_0}(V):=\frac{\ker D_0(V)}{\ker D_0(V)\cap \operatorname{im}\hspace{0.5mm}D_0(V)},
\]
where $D_0(V)$ is the operator by which $D_0$ acts on $V\otimes_{\C}S_0$, with $S_0$ a spin module for $\operatorname{Cl}(\fp_0,\beta_0)$, and $\widetilde{K}$ is the relevant spin double cover of $K$.

Following the reductive case, we define the Dirac series of $G(\R)_0$ to consist of the equivalence classes of irreducible infinitesimally unitary $(\fg_0,K)$-modules with nonzero Dirac cohomology. Our first main result, Theorem \ref{DCCMG}, determines the Dirac series completely.

\begin{theorem*}[Theorem \ref{DCCMG}]
An irreducible infinitesimally unitary $(\fg_0,K)$-module belongs to the Dirac series of $G(\R)_0$ if and only if $\fp_0$ acts trivially on it. Every such module is finite-dimensional.
\end{theorem*}

In other words, the Dirac series of $G(\R)_0$ consists precisely of the irreducible representations pulled back from $K(\R)$ along the canonical quotient $G(\R)_0\longrightarrow K(\R)$. The problem of determining the Dirac series for  real reductive groups recently drew considerable attention, see e.g., \cite{zbMATH07209584,zbMATH07357500,dongwong,zbMATH07702898,zbMATH08086149}.

Our second main result is an analogue of Vogan's conjecture for Cartan motion groups. It is, in fact, stronger than its reductive counterpart. Whereas Vogan’s conjecture for $G(\mathbb R)$ asserts that Dirac cohomology determines the infinitesimal character of a module in the Dirac series, for $G(\mathbb R)_0$  the following holds.  

\begin{theorem*}[Theorem \ref{sss271}]
Let $V$ be a $(\fg_0,K)$-module in the Dirac series of $G(\R)_0$.
Then the isomorphism class of $V$ is completely determined by its
Dirac cohomology $H_{D_0}(V)$.
\end{theorem*}

In Section \ref{DS} we prove Theorem \ref{DCCMG}. The proof uses the explicit realization of the irreducible $(\fg_0,K)$-modules on spaces of functions on $K(\R)$, following the description of the admissible dual of $G(\R)_0$ from \cite{delormeChampetier,GaudillotEstrada2026CovariantRepresentations}. In addition, for a Cartan motion group of a real reductive group of real rank one, we go beyond the unitary case: in Theorem \ref{33} we show that an irreducible $(\fg_0,K)$-module, not assumed infinitesimally unitary, has nonzero Dirac cohomology if and only if $\fp_0$ acts trivially on it. In particular, in the real rank one case, the Dirac series is not enlarged by dropping the unitarity assumption.

In Section \ref{thedeformationfamily} we change perspective and promote the idea that a Cartan motion group should be thought of as a special fiber of a family. We recall the deformation family $(\g_d,K)$ and the Dirac theory that we developed for it in \cite{afentoulidisalmpanis2025diracoperatorsalgebraicfamilies}, and we construct canonical identifications between the zero fibers of the relevant objects attached to $(\g_d,K)$ and the corresponding objects attached to $G(\R)_0$. Specifically, we obtain a $\widetilde{K}$-equivariant isomorphism of algebras
\[\Phi_A:\boldsymbol{A}_d|_0=\left(\mathcal{U}(\g_d)\otimes_{\C[z]}\operatorname{Cl}(\p_d,\boldsymbol{\beta}_{d})\right)\Big|_0\longrightarrow A_0\]
carrying the zero fiber of the Dirac operator $D(\g_d,\boldsymbol{\beta}_d)$ of the family to $D_0$, and the zero fiber of the Casimir $\Omega_d$ of the family to $\Omega_0$. In this precise sense, the intrinsic Dirac theory of $G(\R)_0$ developed in Sections \ref{DiracTheory} and \ref{DS} is the zero fiber of the Dirac theory of the deformation family.

Section \ref{DiracCohomologyAndSpecialization} studies the relation between Dirac cohomology and specialization at zero. For an algebraic family of Harish-Chandra modules $\boldsymbol{V}$ for $(\g_d,K)$, we first show that the specialization at zero of the action of the family Dirac operator computes the Dirac cohomology of the $(\fg_0,K)$-module $\boldsymbol{V}|_0$. We then construct a canonical $\widetilde{K}$-equivariant comparison morphism
\[
\mu_{\boldsymbol V}:
H_{D(\g_d,\boldsymbol\beta_d)}(\boldsymbol V)|_0
\longrightarrow H_{D_0}(\boldsymbol{V}|_0)
\]
from the zero fiber of the Dirac cohomology of the family into the Dirac cohomology of its zero fiber. We give a sufficient condition for $\mu_{\boldsymbol V}$ to be injective; in particular, $\mu_{\boldsymbol V}$ is injective whenever $\boldsymbol{V}|_0$ is infinitesimally unitary. Finally, we present an example in which $\mu_{\boldsymbol V}$ is injective but not surjective, showing that Dirac cohomology need not commute with specialization at zero.

\subsection*{Acknowledgments}
This research was supported by the Israel Science Foundation (grant No. 1040/22).

\section{Dirac theory for Cartan motion groups}\label{DiracTheory}
This section develops the algebraic Dirac theory of a Cartan motion group. We specify the relevant Clifford algebra, construct the Dirac operator and establish its square formula, define Dirac cohomology, describe the associated Dirac series (the proof is given in Section \ref{DS}), and prove an analog of Vogan's conjecture.
Some of the background constructions parallel the reductive-group preliminaries in \cite[Sec.~2]{afentoulidisalmpanis2025diracoperatorsalgebraicfamilies}; here they are reformulated intrinsically for the Cartan motion group and adapted to its non-reductive Lie algebra.

\subsection{The Cartan motion group of a real reductive group}\label{Gr}
In this subsection we specify the basic assumptions and establish running notation for the Cartan motion group of a real reductive group.
\subsubsection{The reductive group}
Let $G(\R)$ be a connected, noncompact real reductive group in the sense of \cite{MR4146144}. Thus $G(\R)$  is the fixed point set $G^{\sigma}$ of a connected complex reductive algebraic group $G$ endowed with an antiholomorphic involution $\sigma$. Choose a corresponding Cartan involution $\theta$ of $G$ and put
\[
K:=G^{\theta}, \qquad K(\R):=K^{\sigma}.
\]
Then $K$ is a complex reductive algebraic group, while $K(\R)$ is a maximal compact subgroup of $G(\R)$. Write
\[
\fg=\operatorname{Lie}(G), \qquad \fk=\operatorname{Lie}(K)
\]
for the complex Lie algebras, and
\[
\fg^{\sigma}=\operatorname{Lie}(G^{\sigma}), \qquad
\fk^{\sigma}=\operatorname{Lie}(K^{\sigma})
\]
for their real forms. The $\pm1$ eigenspace decompositions for $\theta$, also known as Cartan decompositions, are
\[
\fg=\fk\oplus\fp,\qquad
\fg^{\sigma}=\fk^{\sigma}\oplus\fp^{\sigma},
\]
where $\fk=\fg^{\theta}$, $\fp=\fg^{-\theta}$, and $\fp^{\sigma}=\fp\cap\fg^{\sigma}$. We write $\fg'=[\fg,\fg]$ for the derived Lie algebra of $\fg$.

\subsubsection{The Cartan motion group} 
Restricting the adjoint representation gives an action of $K(\R)$ on the vector space $\fp^{\sigma}$. The associated semidirect product
\[
G(\R)_0:=\fp^{\sigma}\rtimes K(\R)
\]
is called the \textit{Cartan motion group of $G(\R)$}.

\begin{remark*}
Any two Cartan involutions of $G(\R)$ give rise to isomorphic Cartan motion groups.
\end{remark*}

The  Lie algebra of $G(\R)_0$ is the semidirect product
\[
\fg^{\sigma}_0=\fp^{\sigma}\rtimes\fk^{\sigma}.
\]
On its underlying vector space $\fp^{\sigma}\times\fk^{\sigma}$, the bracket is given by
\[
[(X_1,Y_1),(X_2,Y_2)]_0
=([Y_1,X_2]+[X_1,Y_2],[Y_1,Y_2]),
\qquad (X_i,Y_i)\in\fp^{\sigma}\times\fk^{\sigma}.
\]
Here the action of $\fk^{\sigma}$ on $\fp^{\sigma}$ is the restriction of the adjoint action of $\fg^{\sigma}$.
The map
\[
I:\fg^{\sigma}_0\longrightarrow\fg^{\sigma},\qquad I(X,Y)=X+Y,
\]
with $(X,Y)\in \fp^{\sigma}\times\fk^{\sigma}$
is a canonical vector-space isomorphism. We use it throughout to identify the two underlying vector spaces. Accordingly,
\[
\fg^{\sigma}_0=\fk^{\sigma}\oplus\fp^{\sigma}_0,\qquad
\fp^{\sigma}_0=\fp^{\sigma},
\]
as $\fk^{\sigma}$-modules, and the Lie bracket of $\fg^{\sigma}_0$ when realized on the vector space $\fg^{\sigma}$ via $I$ is given by 
\[
[X_1+Y_1,X_2+Y_2]_0
=[Y_1,X_2]+[X_1,Y_2]+[Y_1,Y_2]
\]
for $X_i\in\fp^{\sigma}$ and $Y_i\in\fk^{\sigma}$. In particular, $\fp^{\sigma}_0$ is abelian.

After complexification we obtain
\[
G_0=\fp\rtimes K,\qquad
\fg_0=\fp_0\rtimes\fk,\qquad \fp_0=\fp.
\]
The semidirect products use the restricted adjoint actions. We continue to identify $\fg_0$ and $\fg$ as vector spaces, while distinguishing their Lie brackets by writing $[\ ,\ ]_0$ for the bracket of $\fg_0$.

\subsection{Clifford algebra for Cartan motion groups}
In this subsection we explain how the Clifford algebra associated with the reductive Lie algebra $\fg$ also serves as the Clifford algebra associated with $\fg_0$, the Lie algebra of the Cartan motion group.

\subsubsection{The invariant symmetric form on $\fp_0$}\label{invform}

The Lie algebra $\fg_0$ does not, in general, admit a nondegenerate
invariant symmetric bilinear form. For the construction of the
Clifford algebra and the Dirac operator below, however, it is enough
to have such a form on $\fp_0$.

Choose, once and for all, a symmetric bilinear form $\beta$ on $\fg$
satisfying the following conditions.
\begin{enumerate}[(1)]
\item The restriction of $\beta$ to $\fg'=[\fg,\fg]$ is the Killing
form of $\fg'$.
\item The restriction of $\beta$ to $\fp^\sigma$ is positive definite,
and the restriction of $\beta$ to $\fk^\sigma$ is negative definite.
\item The form $\beta$ is $\theta$-invariant, that is,
$\beta(\theta X,\theta Y)=\beta(X,Y)$ for all $X,Y\in\fg$;
equivalently, $\fk$ and $\fp$ are $\beta$-orthogonal.
\item The form $\beta$ is invariant under both $K$ and $\fg$;
explicitly,
\begin{align*}
\beta(\operatorname{Ad}(k)X_1,\operatorname{Ad}(k)X_2)
&=\beta(X_1,X_2),
&&X_1,X_2\in\fg,\quad k\in K,\\
\beta([Z,X],Y)+\beta(X,[Z,Y])
&=0,
&&X,Y,Z\in\fg.
\end{align*}
\end{enumerate}
Conditions (2) and (3) imply that $\beta$ is nondegenerate on $\fg$.

Such forms exist. Indeed, let $\mathfrak{z}(\fg)$ be the center of
$\fg$, so that $\fg=\mathfrak{z}(\fg)\oplus\fg'$, both summands being
stable under $\theta$, $\sigma$, $K$, and $\fg$. For any
$\fg$-invariant symmetric bilinear form on $\fg$, the subspaces
$\mathfrak{z}(\fg)$ and $\fg'$ are orthogonal, and the Killing form
of $\fg'$ satisfies (2)--(4) on $\fg'$. Hence one may take $\beta$
to be the orthogonal direct sum of the Killing form of $\fg'$ and any
symmetric bilinear form on
$\mathfrak{z}(\fg)
=(\mathfrak{z}(\fg)\cap\fk)\oplus(\mathfrak{z}(\fg)\cap\fp)$
for which the two summands are orthogonal, and which is negative
definite on $\mathfrak{z}(\fg)\cap\fk^\sigma$ and positive definite
on $\mathfrak{z}(\fg)\cap\fp^\sigma$; such a form on
$\mathfrak{z}(\fg)$ is automatically $K$- and $\fg$-invariant,
because $K$ and $\fg$ act trivially on $\mathfrak{z}(\fg)$.
Condition (3) is not needed in this section; it will be used in
Section \ref{thedeformationfamily}, where the Casimir element of
$\fg$ is decomposed according to $\fg=\fk\oplus\fp$.

Using the vector-space identification $\fp_0=\fp$, define
\[
\beta_0
:=
\beta|_{\fp_0\times\fp_0}.
\]

\begin{lemma*}
The form $\beta_0$ is a nondegenerate
$(\fg_0,K)$-invariant symmetric bilinear form on $\fp_0$.
\end{lemma*}

\begin{proof}
The form $\beta_0$ is symmetric because $\beta$ is symmetric. Since
$\beta$ is positive definite on the real form
$\fp_0^\sigma=\fp^\sigma$, its complexification $\beta_0$ is
nondegenerate.

The $K$-invariance of $\beta_0$ follows immediately from the
$K$-invariance of $\beta$ and the fact that $\fp_0$ is $K$-stable.

It remains to verify invariance under $\fg_0$. Let
$Z\in\fg_0$ and $X_1,X_2\in\fp_0$. Decompose  $Z=Y+X$ with $Y\in\fk$ and $X\in\fp_0$.
Since $\fp_0$ is abelian in $\fg_0$, we have
\[
[Z,X_i]_0=[Y,X_i]_0=[Y,X_i],
\qquad i=1,2.
\]
Therefore, by the $\fg$-invariance of $\beta$,
\begin{align*}
\beta_0([Z,X_1]_0,X_2)
+\beta_0(X_1,[Z,X_2]_0)
&=
\beta([Y,X_1],X_2)
+\beta(X_1,[Y,X_2])=0.
\end{align*}
Thus $\beta_0$ is invariant under $\fg_0$.
\end{proof}

When the center of $\fg$ has a nonzero intersection with $\fp$, the
form $\beta_0$ may depend on the chosen extension $\beta$ of the
Killing form. Accordingly, throughout the paper all Clifford
algebras, Casimir elements, and Dirac operators are understood to be
canonical relative to this fixed choice of $\beta$. In the
semisimple case, where $\beta$ is the Killing form, no such choice is
involved.

\subsubsection{The Clifford algebra}\label{2.2.2}
We associate with $\fg_0$ \emph{the same} Clifford algebra that is associated with $\fg$. 
Explicitly, the Clifford algebra is denoted by  
$\operatorname{Cl}(\mathfrak{p}_0,\beta|_{\fp_0})=\operatorname{Cl}(\mathfrak{p},\beta|_{\fp})$ and we realize it as  the quotient of the tensor algebra $T(\mathfrak{p}_0)$ by the two-sided ideal $I(\fp_0,\beta_0)$ generated by all elements of the form
\begin{equation*}
	X\otimes Y+Y\otimes X-\beta_0(X,Y), \quad X,Y\in\mathfrak{p}_0.
\end{equation*} 
We denote   the canonical embedding of $\mathfrak{p}_0$ into $\operatorname{Cl}(\mathfrak{p}_0,\beta_0)$ by $\gamma_0=\gamma_{\fp_0,\beta_0}$. 
 
Below we mention known properties of $\operatorname{Cl}(\mathfrak{p}_0,\beta_0)$ that will be used throughout the text. For more information see \cite{goodman,pandzic}. We  shall follow the notation of \cite{afentoulidisalmpanis2025diracoperatorsalgebraicfamilies}.

\subsubsection{The embedding  of $\fk$ in the Clifford algebra}\label{223}
For $a,b\in\fp_0$, define
\begin{equation*}
R_{\beta_0,a,b}(v)=\beta_0(b,v)a-\beta_0(a,v)b,
\qquad v\in\fp_0.
\end{equation*}
These  operators span $\mathfrak{so}(\fp_0,\beta_0)$; see \cite[Lemma~6.2.1]{goodman}. Moreover, there is a unique injective Lie algebra homomorphism
\[
\varphi_{\beta_0}:\mathfrak{so}(\fp_0,\beta_0)\longrightarrow
\operatorname{Cl}(\fp_0,\beta_0)
\]
such that
\[
\varphi_{\beta_0}(R_{\beta_0,a,b})
=\frac12[\gamma_0(a),\gamma_0(b)].
\]

The action of $\fk$ on $\fp_0$ preserves $\beta_0$, and hence defines a Lie algebra homomorphism $\operatorname{ad}_0:\fk\longrightarrow\mathfrak{so}(\fp_0,\beta_0)$. For $Y\in\fk$ and $X\in\fp_0$, this action is given by $\operatorname{ad}_0(Y)(X)=[Y,X]_0=[Y,X]$.

Below we  use the composite map
\[
\alpha_0=\alpha_{\beta_0}:=
\varphi_{\beta_0}\circ\operatorname{ad}_0:
\fk\longrightarrow \operatorname{Cl}(\fp_0,\beta_0).
\] 

\subsection{The generalized pair for Cartan motion groups}\label{genpair}
Set
\[
A_0=A(\fg_0,\beta_0)
:=\mathcal{U}(\fg_0)\otimes_{\C}\operatorname{Cl}(\fp_0,\beta_0).
\]
The map
\begin{equation*}
\Delta_{\beta_0}:\fk\longrightarrow A_0,\qquad
\Delta_{\beta_0}(X)=X\otimes1+1\otimes\alpha_{\beta_0}(X),
\end{equation*}
is an injective Lie algebra homomorphism. We refer to it as the \textit{diagonal embedding}.

The group $K$ acts diagonally on the two factors of $A_0$. Differentiating this action gives the inner derivation
\[
X\longmapsto[\Delta_{\beta_0}(X),\_],
\qquad X\in\fk.
\]
Consequently, $A_0$, the $K$-action, and the map $\Delta_{\beta_0}$ form a generalized pair in the sense of \cite[Ch.~I.6]{KNV}. We write $\fk_{\Delta_0}:=\Delta_{\beta_0}(\fk)$ and denote the corresponding copy of $\mathcal{U}(\fk)$ in $A_0$ by $\mathcal{U}(\fk_{\Delta_0})$.

\subsection{Dirac operator for Cartan motion groups}
In this subsection we give a canonical, coordinate-independent definition of the Dirac operator of a Cartan motion group. We also define a canonical Casimir-type element in the center of the universal enveloping algebra of $\fg_0$, and prove that the square of the Dirac operator is proportional to this element.
\subsubsection{Canonical elements }\label{CanEl} 
Let $\fl\subseteq\fp_0$ be a linear subspace on which $\beta_0$ is nondegenerate. The restriction of the form identifies $\fl$ with its dual through
\[
\check{\beta}_0|_{\fl}:\fl\longrightarrow\fl^*,\qquad
X\longmapsto\beta_0(X,\_).
\]
Combining the inverse of this identification with the canonical isomorphism $\operatorname{End}(\fl)
\cong\fl^*\otimes_{\C}\fl$ yields
\begin{equation*}
\operatorname{End}(\fl)
\cong\fl^*\otimes_{\C}\fl
\cong\fl\otimes_{\C}\fl.
\end{equation*}
We define $\omega(\fl,\beta_0)$ to be the image in $\fl\otimes_{\C}\fl$ of the identity endomorphism $\operatorname{Id}_{\fl}$ under this isomorphism.

If $\{e_1,\ldots,e_m\}$ is a basis of $\fl$ and $\{e'_1,\ldots,e'_m\}$ is its $\beta_0$-dual basis, so that $\beta_0(e_i,e'_j)=\delta_{ij}$, then
\[
\omega(\fl,\beta_0)
=\sum_{i=1}^{m}e'_i\otimes e_i
=\sum_{i,j=1}^{m}\beta_0(e'_j,e'_i)e_j\otimes e_i.
\]

\subsubsection{Casimir for  Cartan motion groups}\label{s242}
\begin{definition*}
The \textit{Casimir element} of $\fg_0$ associated with $(\fp_0,\beta_0)$ is the image
\[
\Omega_0=\Omega(\fp_0,\beta_0)\in\mathcal{U}(\fg_0)
\]
of $\omega(\fp_0,\beta_0)$ under the composite
\[
\fp_0\otimes_{\C}\fp_0
\longrightarrow T(\fp_0)
\longrightarrow T(\fg_0)
\longrightarrow\mathcal{U}(\fg_0).
\]
\end{definition*}

\begin{lemma*}
The element $\Omega(\fp_0,\beta_0)$ lies in $\mathcal{Z}(\fg_0)^K$, where $\mathcal{Z}(\fg_0)$ denotes the center of $\mathcal{U}(\fg_0)$.
\end{lemma*}

\begin{proof}
Let
\[
m:\operatorname{End}(\fp_0)\longrightarrow\mathcal{U}(\fg_0)
\]
denote the composite morphism  used in the definition of the Casimir. The $K$-invariance of $\beta_0$ makes the identification $\fp_0\simeq\fp_0^*$, and hence $m$, $K$-equivariant. Since the identity endomorphism $\operatorname{Id}_{\fp_0}$ is fixed by $K$, its image $\Omega(\fp_0,\beta_0)$ is also $K$-fixed.

It remains to prove that the Casimir is central. Differentiating $K$-equivariance shows that $m$ is $\fk$-equivariant. On the other hand, $\fp_0$ is abelian hence  its action on $\operatorname{End}(\fp_0)$ is trivial.  In addition  $m(\operatorname{End}(\fp_0))$ is contained in $\mathcal{U}(\fp_0)$. Thus $\fp_0$ acts trivially on the image of $m$ by commutators. It follows that $m$ is $\fg_0$-equivariant. Because $\operatorname{Id}_{\fp_0}$ is fixed by the induced $\fg_0$-action, we obtain
\[
[\fg_0,\Omega(\fp_0,\beta_0)]=0.
\]
Therefore $\Omega(\fp_0,\beta_0)\in\mathcal{Z}(\fg_0)^K$.
\end{proof} 

\subsubsection{Canonical elements and Casimirs in general}\label{s243} 
In Subsection \ref{s44} we shall use canonical elements and Casimirs for Lie algebras different from $\fg_0$. For that purpose we note here that  for any complex Lie algebra $\mathfrak{m}$ equipped with a symmetric bilinear form $\beta_{\mathfrak{m}}$, and any vector subspace $\mathfrak{l}_\mathfrak{m}$ of $\mathfrak{m}$ on which the restriction of $\beta_{\mathfrak{m}}$ is nondegenerate, similarly to Subsubsection \ref{CanEl}, we can define \[\omega(\mathfrak{l}_\mathfrak{m},\beta_{\mathfrak{m}})\in \mathfrak{l}_\mathfrak{m}\otimes \mathfrak{l}_\mathfrak{m}.\]
In addition, similarly to Subsubsection \ref{s242} we can define 
\[
\Omega(\mathfrak{l}_\mathfrak{m},\beta_{\mathfrak{m}})\in\mathcal{U}(\mathfrak{m}).
\]
If in addition  $\mathfrak{l}_\mathfrak{m}$ is an ideal of $\mathfrak{m}$ and the form $\beta_{\mathfrak{m}}$ is $\operatorname{ad}_{\mathfrak{m}}$-invariant on $\mathfrak{l}_\mathfrak{m}$ in the sense that 
\[\beta_{\mathfrak{m}}([x,y],z)=\beta_{\mathfrak{m}}(x,[y,z]),\quad \forall x,z\in \mathfrak{l}_\mathfrak{m}, y\in \mathfrak{m}, \]
then 
\[
\Omega(\mathfrak{l}_\mathfrak{m},\beta_{\mathfrak{m}})\in\mathcal{Z}(\mathfrak{m}).
\]
When $\mathfrak{l}_\mathfrak{m}=\mathfrak{m}$, with $\mathfrak{m}$ semisimple and $\beta_{\mathfrak{m}}$ the Killing form, $\Omega(\mathfrak{m},\beta_{\mathfrak{m}})$ is the canonical Casimir element of $\mathfrak{m}$.
\subsubsection{ Dirac operator for Cartan motion groups}\label{243}
The tensor product of the canonical maps from $\fp_0$ to $\mathcal{U}(\fg_0)$ and from $\fp_0$ to $\operatorname{Cl}(\fp_0,\beta_0)$ induces a $K$-equivariant linear embedding
\begin{equation}\label{Endembedd}
\fp_0\otimes_{\C}\fp_0
\hookrightarrow
\mathcal{U}(\fg_0)\otimes_{\C}\operatorname{Cl}(\fp_0,\beta_0)=A_0.
\end{equation}

\begin{definition*}\label{complexDiracoperator}
The \textit{algebraic Dirac operator} of $G(\R)_0$ is the image 
\[
D_0=D_{\fg_0,\beta_0}=D_{\beta_0}\in A_0
\]
of $\omega(\fp_0,\beta_0)$ under \eqref{Endembedd}.
\end{definition*}

For any pair of $\beta_0$-dual bases $\{e_i\}$ and $\{e'_i\}$ of $\fp_0$, one has
\[
D_0
=\sum_i e_i\otimes\gamma_0(e'_i)
=\sum_{i,j}\beta_0(e'_i,e'_j)e_i\otimes\gamma_0(e_j).
\]

\begin{lemma*}[The square of the Dirac operator]
In $A_0$,
\[
2D_0^2
=\Omega(\fp_0,\beta_0)\otimes1_{\operatorname{Cl}(\fp_0,\beta_0)}.
\]
\end{lemma*}

\begin{proof}
Because $\beta_0$ is positive definite on $\fp^{\sigma}_0$, choose a $\beta_0$-orthonormal basis $\{e_1,\ldots,e_m\}$ of that real vector space. Then
\[
D_0=\sum_{i=1}^{m}e_i\otimes\gamma_0(e_i).
\]
Using the commutativity of $\fp_0$ inside $\mathcal{U}(\fg_0)$, group the diagonal and off-diagonal terms in the square:
\begin{align*}
2D_0^2
&=2\sum_{i=1}^{m}e_i^2\otimes\gamma_0(e_i)^2\\
&\quad
+2\sum_{1\leq i<j\leq m}e_ie_j\otimes
\bigl(\gamma_0(e_i)\gamma_0(e_j)
+\gamma_0(e_j)\gamma_0(e_i)\bigr).
\end{align*}
The Clifford relation gives
\[
\gamma_0(e_i)\gamma_0(e_j)+\gamma_0(e_j)\gamma_0(e_i)
=\beta_0(e_i,e_j);
\]
in particular $\gamma_0(e_i)^2=\tfrac12$.
Orthonormality therefore eliminates the off-diagonal terms and yields
\[
2D_0^2
=\sum_{i=1}^{m}e_i^2\otimes1_{\operatorname{Cl}(\fp_0,\beta_0)}
=\Omega(\fp_0,\beta_0)\otimes1_{\operatorname{Cl}(\fp_0,\beta_0)}.
\]
\end{proof}

\subsection{Dirac cohomology for Cartan motion groups}
In this subsection we define the Dirac cohomology of a $(\fg_0,K)$ module.

\subsubsection{The spin module}\label{2.5.1}
Since $\operatorname{Cl}(\fp_0,\beta_0)=\operatorname{Cl}(\fp,\beta|_{\fp})$,
we use the same spin modules for the reductive group and for its Cartan motion group. By a spin module for $\operatorname{Cl}(\fp_0,\beta_0)$ we mean a simple module over this Clifford algebra.

Suppose first that $\dim\fp_0$ is even. Choose a polarization
\[
\fp_0=\fp_0^+\oplus\fp_0^-
\]
into dual maximal isotropic subspaces. A spin module is then realized as
\[
S_0=S_0(\fp_0,\beta_0,\fp_0^-)
=\bigwedge\fp_0^-.
\]
The Clifford action is determined on the two isotropic summands as follows:
\begin{enumerate}
\item if $X\in\fp_0^-$ and $Y\in S_0$, then
\[
\gamma_0(X)Y=X\wedge Y;
\]
\item if $X\in\fp_0^+$, then $\gamma_0(X)$ is the graded derivation of degree $-1$ characterized by
\[
\gamma_0(X)Y=\beta_0(X,Y),
\qquad Y\in\fp_0^-\subseteq\bigwedge\fp_0^-.
\]
\end{enumerate}
We denote the resulting representation by
\[
\gamma'_{\fp_0^-,\beta_0}:
\operatorname{Cl}(\fp_0,\beta_0)\longrightarrow\operatorname{End}(S_0).
\]
When $\dim\fp_0=2m+1$ is odd, choose an orthogonal decomposition
\[
\fp_0=\fp_0^+\oplus\fp_0^-\oplus \C e_0,
\]
where $\fp_0^\pm$ are dual maximal isotropic subspaces and
$\beta_0(e_0,e_0)=2$.  Fix $\varepsilon\in\{1,-1\}$ and put
\[
S_0=\bigwedge\fp_0^-.
\]
The action of $\fp_0^+\oplus\fp_0^-$ is the same as above, while
$\gamma_0(e_0)$ acts by multiplication by $\varepsilon$ on
$\bigwedge^{\mathrm{even}}\fp_0^-$ and by multiplication by
$-\varepsilon$ on $\bigwedge^{\mathrm{odd}}\fp_0^-$.  This defines
one of the two simple $\operatorname{Cl}(\fp_0,\beta_0)$-modules.  We fix this choice
for the remainder of the paper.

\subsubsection{The spin double cover}\label{dc}
The maps introduced above give $S_0$ a $\fk$-module structure through the composite
\begin{equation*}
\fk
\overset{\operatorname{ad}_0}{\longrightarrow}
\mathfrak{so}(\fp_0,\beta_0)
\overset{\varphi_{\beta_0}}{\longrightarrow}
\operatorname{Cl}(\fp_0,\beta_0)
\overset{\gamma'_{\fp_0^-,\beta_0}}{\longrightarrow}
\operatorname{End}(S_0).
\end{equation*}
This infinitesimal action integrates to the complex reductive group $\widetilde K$, obtained by complexifying the spin double cover of $K(\R)$. We refer to \cite[Sec.~3.2.1]{pandzic} and \cite[Sec.~2.0.12]{afentoulidisalmpanis2025diracoperatorsalgebraicfamilies} for the pullback construction of this cover.

Let
\[
\pi_{\widetilde K,S_0}:\widetilde K
\longrightarrow\operatorname{Aut}_{\C}(S_0)
\]
denote the resulting representation. Its differential is
\[
\gamma'_{\fp_0^-,\beta_0}\circ
\varphi_{\beta_0}\circ\operatorname{ad}_0.
\]
With the action induced from $K$, the triple consisting of $A_0$, $\widetilde K$, and the diagonal embedding $\Delta_{\beta_0}$ is again a generalized pair.

  For later use we denote the covering map $\widetilde{K}\longrightarrow K$ by $\mathcal{C}$.

\subsubsection{Dirac cohomology for Cartan motion groups}
Given a $(\mathfrak{g}_0,K)$-module $V$,  we denote the corresponding representations by 
\[\pi_{\mathcal U(\fg_0),V}:
\mathcal U(\fg_0)\longrightarrow\operatorname{End}_{\C}(V)  \quad\text{and} \quad \pi_{K,V}:K\longrightarrow\operatorname{Aut}_{\C}(V). \]
The vector space  $V\otimes_{\C}S_0$ is an $(A_0,\widetilde{K})$-module via
\begin{align*} 
\pi_{ A_0, V}&:A_0\longrightarrow \operatorname{End}_{\C}(V\otimes_{\C}S_0) \\
 \pi_{A_0,V}&(u\otimes c)
=
\pi_{\mathcal U(\fg_0),V}(u)
\otimes
\gamma'_{\fp_0^-,\beta_0}(c)
\end{align*}
for $u\in\mathcal U(\fg_0)$,
$c\in\operatorname{Cl}(\fp_0,\beta_0)$, and 
\begin{align*} 
\pi_{ \widetilde{K}, V}&: \widetilde{K}\longrightarrow \operatorname{Aut}_{\C}(V\otimes_{\C}S_0) \\
\pi_{ \widetilde{K}, V}&(k)
=\pi_{K, V}(\mathcal{C}(k)) \otimes \pi_{\widetilde K,S_0}(k),
\end{align*}
for $k\in \widetilde{K}$.

In particular the Dirac operator acts on  $V\otimes_{\C}S_0$ via a $\widetilde{K}$-equivariant linear operator that we denote by $D_0(V)=D_{\fg_0,\beta_0}(V)=D_{\beta_0}(V)$.  

\begin{definition*} Let $V$ be a $(\mathfrak{g}_0,K)$-module. The \textit{Dirac cohomology} of $V$ is the $\widetilde{K}$-module 
	\begin{equation*}
		H_{D_0}(V):=\frac{\ker D_0(V)}{\ker D_0(V)\cap \operatorname{im}\hspace{0.5mm}D_0(V)}.
	\end{equation*}
\end{definition*}

\subsection{Dirac series for Cartan motion groups}\label{DCCMG}
In this subsection we give a full classification of the Dirac series  of a Cartan motion group.

Recall that the Dirac series of a real reductive group consists of the equivalence classes of irreducible, infinitesimally unitary $(\fg,K)$-modules whose Dirac cohomology is nonzero; for this terminology see, e.g., \cite{Dong2020IMRN,dong1,dongwong}. We use the same criterion to define the Dirac series of $G(\R)_0$, with $(\fg,K)$ replaced by $(\fg_0,K)$.

\begin{theorem*}
An irreducible infinitesimally unitary $(\fg_0,K)$-module belongs to the Dirac series of $G(\R)_0$ if and only if $\fp_0$ acts trivially on it. Every such module is finite-dimensional.
\end{theorem*}

  The theorem will be proved in Section~\ref{DS} using the standard realization of irreducible $(\fg_0,K)$-modules on spaces of functions on $K$.

 \begin{remark*}
  The irreducible infinitesimally unitary $(\fg_0,K)$-modules annihilated by $\fp_0$ exhaust all finite-dimensional irreducible $(\fg_0,K)$-modules precisely when
\[
Z(G(\R)_0)\cap\fp_0^{\sigma}=\{0\},
\]
where $Z(G(\R)_0)$ denotes the center of the Cartan motion group.   
 \end{remark*}

\subsection{Vogan's conjecture for Cartan motion groups}\label{VCCMG}

For a reductive group, Vogan's conjecture, proved in
\cite{huangpandzic}, asserts that the infinitesimal character of an
irreducible infinitesimally unitary $(\fg,K)$-module with nonzero
Dirac cohomology is determined by its Dirac cohomology. 
For a module $V$ in the Dirac series of $G(\mathbb R)_0$, we prove a stronger statement: the $\widetilde K$-module $H_{D_0}(V)$ determines the isomorphism class of $V$. Moreover, using its compatible Clifford algebra action, we give an explicit reconstruction of $V$.

\subsubsection{The $\widetilde{K}$-module $H_{D_0}(V)$ determines the isomorphism class of $V$}\label{sss271}
\begin{theorem*}
Let $V$ be a $(\fg_0,K)$-module in the Dirac series of $G(\R)_0$.
Then the isomorphism class of $V$ is completely determined by its
Dirac cohomology $H_{D_0}(V)$.
\end{theorem*}

\begin{proof}
By Theorem~\ref{DCCMG}, the module $V$ is finite-dimensional and
$\fp_0$ acts trivially on it. Consequently, 
\[
H_{D_0}(V)=V\otimes_{\C}S_0
\]
as $\widetilde K$-modules, where $\widetilde K$ acts on $V$ through
the covering map $\mathcal C:\widetilde K\longrightarrow K$ composed with the action of $K$ on $V$.

For a finite-dimensional group representation $E$, write
$\chi_E(g)=\operatorname{tr}(g|_E)$ for its character. Taking
characters in the preceding identity gives
\[
\chi_{H_{D_0}(V)}(\widetilde k)
=
\chi_V\bigl(\mathcal C(\widetilde k)\bigr)
\chi_{S_0}(\widetilde k),
\qquad \widetilde k\in\widetilde K.
\]
Since $\chi_{S_0}(1)=\dim S_0\ne0$, there is an open neighborhood
$U$ of the identity in $\widetilde K$ on which $\chi_{S_0}$ is
nonzero. Thus
\[
\chi_V\bigl(\mathcal C(\widetilde k)\bigr)
=
\frac{\chi_{H_{D_0}(V)}(\widetilde k)}
     {\chi_{S_0}(\widetilde k)},
\qquad \widetilde k\in U.
\]
The $\widetilde K$-module $H_{D_0}(V)$ therefore determines $\chi_V$
on the open neighborhood $\mathcal C(U)$ of the identity in $K$.

Our standing assumption that $G(\R)$ is connected implies that
$K(\R)$, and hence its complexification $K$, is connected.
Characters of finite-dimensional algebraic $K$-modules are analytic,
so their values on a neighborhood of the identity determine them on
all of $K$. Since $K$ is reductive, its finite-dimensional algebraic
representations are completely reducible and are determined up to
isomorphism by their characters. Hence the underlying $K$-module of
$V$ is the unique such module whose character satisfies the displayed
quotient formula. Finally, since $\fp_0$ acts trivially on $V$, this
also determines its isomorphism class as a $(\fg_0,K)$-module.
\end{proof}

\subsubsection{Reconstruction of $V$ from the $\widetilde{K}$-equivariant  $\operatorname{Cl}(\fp_0,\beta_0)$-module $H_{D_0}(V)$} 
In the last subsubsection we showed that for a   $(\fg_0,K)$-module $V$  in the Dirac series,  the $\widetilde{K}$-module structure of $H_{D_0}(V)$  determines the isomorphism class of $V$. In this subsubsection we show that for a $(\fg_0,K)$-module $V$  in the Dirac series, its  Dirac cohomology is not just  a $\widetilde{K}$-module but also a compatible $\operatorname{Cl}(\fp_0,\beta_0)$-module. In addition, we show that this extra structure of  
$H_{D_0}(V)$ allows one  to reconstruct $V$ from $H_{D_0}(V)$.

For a $(\fg_0,K)$-module $V$ the action of  $A_0 $ on $V\otimes_{\C}S_0$ induces an action of $\operatorname{Cl}(\fp_0,\beta_0)$ on  the same space via the inclusion  
\[\operatorname{Cl}(\fp_0,\beta_0)\simeq \C\otimes_{\C} \operatorname{Cl}(\fp_0,\beta_0)\subset A_0.\]
Moreover, the Clifford and $\widetilde K$-actions satisfy
\[
\widetilde k\cdot(c\cdot w)
=(\widetilde k\cdot c)\cdot(\widetilde k\cdot w)
\]for $\widetilde k\in\widetilde K$, $c\in\operatorname{Cl}(\mathfrak p_0,\beta_0)$, and $w\in V\otimes_{\mathbb C}S_0$. Thus $V\otimes_{\mathbb C}S_0$ is a $\widetilde K$-equivariant $\operatorname{Cl}(\mathfrak p_0,\beta_0)$-module.

If $\fp_0$ acts trivially on $V$ then $H_{D_0}(V)=V\otimes_{\C}S_0$, and hence in this case the Dirac cohomology of $V$ is a $\widetilde{K}$-equivariant $\operatorname{Cl}(\fp_0,\beta_0)$-module.  

Put $\operatorname{Cl}_0:=\operatorname{Cl}(\fp_0,\beta_0)$ if $\dim\fp_0$ is even, and let $\operatorname{Cl}_0:=\operatorname{Cl}^{\bar 0}(\fp_0,\beta_0)$ be the subalgebra of even elements of $\operatorname{Cl}(\fp_0,\beta_0)$ if $\dim\fp_0$ is odd. For a $(\fg_0,K)$-module $V$, the group $\widetilde K$ acts on $\operatorname{Hom}_{\operatorname{Cl}_0}(S_0,V\otimes_{\C}S_0)$ by conjugation , and this action factors through $K$, since the nontrivial element of $\ker\mathcal C$ acts by $-1$ on both $S_0$ and $V\otimes_{\C}S_0$.

\begin{theorem*}
    Let $V$ be a $(\fg_0,K)$-module in the Dirac series of the Cartan motion group $G(\R)_0$. Then $\fp_0$ acts trivially on $V$, so that $H_{D_0}(V)=V\otimes_{\C}S_0$, and the map
\[
\eta_V:V\longrightarrow\operatorname{Hom}_{\operatorname{Cl}_0}\bigl(S_0,H_{D_0}(V)\bigr),
\qquad
\eta_V(v)(s)=v\otimes s,
\]
is an isomorphism of $K$-modules. In particular, $V$ can be constructed  from  its Dirac cohomology, considered as a $\widetilde{K}$-equivariant $\operatorname{Cl}(\fp_0,\beta_0)$-module.
\end{theorem*}

\begin{proof}
The first assertion follows from Theorem \ref{DCCMG}: $\fp_0$ acts trivially on $V$, so $D_0(V)=0$ and $H_{D_0}(V)=V\otimes_{\C}S_0$. It remains to show that $\eta_V$ is an isomorphism of $K$-modules. Let $m\in \mathbb{N}_0$ be such that $\operatorname{dim}(\fp_0)\in\{2m,2m+1\}$. By the classification of complex Clifford algebras, $\operatorname{Cl}_0$ is isomorphic to the simple matrix algebra $M_{2^m}(\C)$, see e.g., \cite[\S6]{GarlingClifford} and \cite[\S20.1]{FultonHarris}; since $\dim S_0=2^m$, it follows that $S_0$ is a simple $\operatorname{Cl}_0$-module.

Note that for $X\in \operatorname{Cl}_0\subseteq \operatorname{Cl}(\mathfrak{p}_0,\beta_0)$, $v\in V$ and $s\in S_0$, 
\[
\eta_V(v)(\gamma'_{\fp_0^-,\beta_0}(X)s)=v\otimes \left(\gamma'_{\fp_0^-,\beta_0}(X)s\right)=\pi_{ A_0, V}(1\otimes X)(v\otimes s),
\]
so indeed $\eta_V(v)\in \operatorname{Hom}_{\operatorname{Cl}_0}\bigl(S_0,V\otimes_{\C}S_0\bigr)$. Because $S_0$ is finite-dimensional, the natural map
\[
V\otimes_{\C}
\operatorname{End}_{\operatorname{Cl}_0}(S_0)
\longrightarrow
\operatorname{Hom}_{\operatorname{Cl}_0}
\bigl(S_0,V\otimes_{\C}S_0\bigr),
\qquad
v\otimes T\longmapsto\bigl(s\mapsto v\otimes T(s)\bigr),
\]
is an isomorphism. By Schur's lemma, $\operatorname{End}_{\operatorname{Cl}_0}(S_0)
=\C\,\operatorname{Id}_{S_0}$. 
Consequently,
\[
\operatorname{Hom}_{\operatorname{Cl}_0}
\bigl(S_0,V\otimes_{\C}S_0\bigr)\cong V,
\]
and  the resulting composition 
\[V \xrightarrow{\eta_V}
\operatorname{Hom}_{\operatorname{Cl}_0}
\bigl(S_0,V\otimes_{\C}S_0\bigr)\cong V,
\] is the identity map of  $V$.
Thus $\eta_V$ is an isomorphism.
It remains to show that $\eta_V$ is $K$-equivariant. 
For $v\in V$, $s\in S_0$, and $\widetilde k\in\widetilde K$, we have
\begin{align*}
\bigl(\widetilde k\cdot\eta_V(v)\bigr)(s)
&=
\pi_{\widetilde K,V}(\widetilde k)\,
\eta_V(v)\bigl(
\pi_{\widetilde K,S_0}(\widetilde k^{-1})s
\bigr)=
\pi_{\widetilde K,V}(\widetilde k)
\bigl(
v\otimes
\pi_{\widetilde K,S_0}(\widetilde k^{-1})s
\bigr)\\
&=
\pi_{K,V}\bigl(\mathcal C(\widetilde k)\bigr)v\otimes s=
\eta_V\!\left(
\pi_{K,V}\bigl(\mathcal C(\widetilde k)\bigr)v
\right)(s).
\end{align*}
So $\eta_V$ is $\widetilde{K}$-equivariant, and hence $K$-equivariant, because the action of $\widetilde K$ on $\operatorname{Hom}_{\operatorname{Cl}_0}\bigl(S_0,V\otimes_{\C}S_0\bigr)$ factors through $K$.

\end{proof}


 \begin{remark*}
The preceding proof essentially contains  a Cartan-motion-group analogue of the spinor
Morita equivalence used by Pand\v{z}i\'c and Renard for reductive
Harish-Chandra pairs; see \cite[Prop.~2.2.1]{MR2778228} for the
even-dimensional case and \cite[Prop.~3.2]{Renard2014DiracSurvey} for
Renard's formulation in both parities. 
\end{remark*}

\section{Determining  the Dirac series of Cartan motion groups}\label{DS}
In this section we prove Theorem \ref{DCCMG}  that determines the Dirac series of a Cartan motion group.  
In addition, for a Cartan motion group of a connected real reductive group of real rank one, we determine which irreducible Harish-Chandra modules (which are not necessarily infinitesimally unitary) have nonzero Dirac cohomology. 

\subsection{The admissible dual of $G(\R)_0$}\label{ad}

In this subsection we describe $\widehat{G(\R)}_0$, the admissible dual of $G(\R)_0$, that is, the set of equivalence classes of irreducible  admissible $(\mathfrak{g}_0,K)$-modules.
The unitary dual $\widehat{G(\R)}_{0,\text{u}}$ of $G(\R)_0$ can  be determined using  the Mackey machine \cite{zbMATH03536358}.  
The more complicated admissible dual $\widehat{G(\R)}_0$ was determined in \cite{delormeChampetier,MR965746} in the case of $G(\R)$ being a semisimple Lie group with a finite center. The more general case of $\widehat{G(\R)}_0$, when $G(\R)$ is a real reductive group, can be found in \cite{GaudillotEstrada2026CovariantRepresentations}.

We fix a maximal abelian subspace $\fa^{\sigma}$ of  $\mathfrak{p}^\sigma$. We denote by $\mathfrak{q}^{\sigma}$ the orthogonal complement of $\fa^{\sigma}$ in $\fp^{\sigma}$ with respect to $\beta$. The corresponding complexifications are denoted by $\fa$ and $\fq$, respectively. 

We shall identify a complex linear functional $\lambda\in \fa^*$ with a complex linear functional on $\fp$ by extending it to be zero  on $\fq$.

We denote by $K^{\lambda}(\R)$ the stabilizer of $\lambda$ in $K(\R)$. The unitary dual $\widehat{K^{\lambda}(\R)}_{\text{u}}$ of $K^{\lambda}(\R)$, that is, the set of equivalence classes of irreducible continuous representations of $K^{\lambda}(\R)$, can be identified with the algebraic dual $\widehat{K^{\lambda}}$  of the complexification $K^{\lambda}$ of $K^{\lambda}(\R)$. By the algebraic dual we mean equivalence classes of irreducible algebraic representations.
Abusing notation we shall treat $\tau\in \widehat{K^{\lambda}}$
as a single representation,  rather than an equivalence class of representations. We shall denote the carrier vector space of $\tau$ by $E^{\tau}$.

Following Afgoustidis \cite{MR4400734}, we shall call  a pair $(\lambda,\tau)\in \fa^*\times \widehat{K^{\lambda}}$ a Mackey datum (for $G(\R)_0$). For such a Mackey datum   we let $V^\infty(\lambda,\tau)$ be the space of smooth maps
\begin{equation*}
	f:K(\R)\rightarrow E^\tau
\end{equation*}
satisfying  
\begin{equation*}
	f(kx)=\tau(x)^{-1}f(k), \quad \forall k\in K(\R), \forall x\in K^\lambda(\R). 
\end{equation*}
The space $V^\infty(\lambda,\tau)$ equipped with the $G(\R)_0$-action 
\begin{equation}\label{theactiononV}
	[(X,k)f](k')\coloneqq e^{i\lambda(k'^{-1}X)}f(k^{-1}k'),\quad \forall (X,k)  \in  \mathfrak{p}_0^\sigma\rtimes K(\R), \forall k'\in K(\R)
\end{equation}
is a representation of $G(\R)_0$. Here, and in what follows, $kX$ stands for $\operatorname{Ad}(k)X$, for $k\in K(\R)$ and $X\in\fp$. We denote by  $V(\lambda,\tau)$  the space of $K$-finite vectors of $V^\infty(\lambda,\tau)$. Then $V(\lambda,\tau)$ is a $(\mathfrak{g}_0,K)$-module. 
We define the Weyl group $W$ of $G(\R)_0$ to be the quotient of the normalizer of $\fa^{\sigma}$ in $K(\R)$ by the centralizer of $\fa^{\sigma}$ in $K(\R)$.

The following combines Thm.~A of \cite{delormeChampetier}
with the discussion in Secs.~3.2--3.4, especially Thm.~15, of
\cite{GaudillotEstrada2026CovariantRepresentations}.

\begin{theorem*}\label{delorme} \quad\\
	\vspace{-5mm}
	\begin{enumerate}
		\item For every $\lambda\in\mathfrak{a}^*$ and every irreducible representation $\tau$ of $K^\lambda$, the $(\mathfrak{g}_0,K)$-module $V(\lambda,\tau)$ is irreducible;
		\item Every irreducible $(\mathfrak{g}_0,K)$-module is isomorphic to $V(\lambda,\tau)$ for some $\lambda\in\mathfrak{a}^*$ and some irreducible representation $\tau$ of $K^\lambda$;
		\item The $(\mathfrak{g}_0,K)$-modules $V(\lambda_1,\tau_1)$ and $V(\lambda_2,\tau_2)$ are equivalent if and only if there exists $w\in W$ such that $\lambda_2=w\lambda_1$ and $\tau_2\cong w\tau_1$.
	\end{enumerate}
\end{theorem*}
Combining the last Theorem with the Mackey machine we   obtain the following corollary. 

\begin{corollary*}
    The irreducible  $(\mathfrak{g}_0,K)$-module $V(\lambda,\tau)$ is infinitesimally unitary if and only if $\lambda(\fa^{\sigma})\subseteq \R$.
\end{corollary*}
\begin{proof}
If $\lambda$ is real-valued on $\fa^\sigma$, then
$V(\lambda,\tau)$ is the space of $K$-finite vectors of the
unitarily induced representation
\[
\operatorname{Ind}_{\fp_0^\sigma\rtimes K^\lambda(\R)}^{G(\R)_0}
\bigl(e^{i\lambda}\otimes\tau\bigr),
\]
and is therefore infinitesimally unitary.

Conversely, suppose that $V(\lambda,\tau)$ is infinitesimally
unitary. By the integration argument in
\cite[Sec.~4.4]{MR4123111}, based on Nelson's theorem
\cite[Theorem~5]{Nelson59}, the Lie algebra action integrates
to an irreducible unitary representation of the universal
covering group of $G(\R)_0$ on the Hilbert space completion
of $V(\lambda,\tau)$.
Since the $\fk^\sigma$-action integrates to the given
$K(\R)$-action, and the inclusion
$K(\R)\hookrightarrow G(\R)_0$ induces an isomorphism
of fundamental groups, this representation descends to
$G(\R)_0$.

By Mackey's classification \cite{zbMATH03536358}, the underlying
$(\fg_0,K)$-module is isomorphic to $V(\lambda',\tau')$ for some
$\lambda'$ that is real-valued on $\fa^\sigma$.
Part~(3) of the preceding theorem implies that
$\lambda=w\lambda'$ for some $w\in W$.
Since $W$ preserves $\fa^\sigma$, the functional $\lambda$
is also real-valued on $\fa^\sigma$.
\end{proof}

 \subsection{The Dirac series of Cartan motion groups}
 In this subsection we prove Theorem \ref{DCCMG}.

\begin{proof}
If $V$ is an infinitesimally unitary irreducible $(\fg_0,K)$-module on which $\fp_0$
acts trivially, then clearly $H_{D_0}(V)=V\otimes_{\C}S_0$ and in particular $V$ belongs to the Dirac series.

On the other hand, if $V$ belongs to the Dirac series of $G(\R)_0$, then by Theorem \ref{ad} there exists a Mackey datum $(\lambda,\tau)\in\mathfrak{a}^*\times \widehat{K^{\lambda}}$ with $V\cong V(\lambda,\tau)$, and by Corollary \ref{ad}, $\lambda$ is real-valued on $\mathfrak{a}^{\sigma}$, hence on $\fp_0^{\sigma}$. We now show that $\lambda\equiv 0$.

For $X\in \mathfrak{p}$ and $f\in V(\lambda,\tau)$, differentiating \eqref{theactiononV} gives
\[
(X f)(k)=i\lambda(k^{-1}X)f(k), \quad \forall k\in K(\R).
\]
Hence, if $\{z_{i}\}$ is a $\beta_0$-orthonormal basis of $\mathfrak{p}_0$,
\[
(\Omega_0f)(k)=\sum_i (z_{i}^2f)(k)=-\sum_i \lambda(k^{-1}z_{i})^2f(k).
\]
Since $\beta_0$ is $K$-invariant, $\{k^{-1}z_{i}\}$ is again a $\beta_0$-orthonormal basis of $\mathfrak{p}_0$, and therefore
\[
\Omega_0 f=-\sum_i\lambda(z_{i})^2f,\qquad \forall f\in V(\lambda,\tau).
\]
Let $x$ be a nonzero vector in $\ker D_0(V)\setminus \operatorname{im}D_0(V)$; such an $x$ exists since $H_{D_0}(V)\neq 0$. By Lemma \ref{243},
\[
0=2D_0^2(x)=\pi_{A_0,V}\bigl(\Omega_0\otimes1_{\operatorname{Cl}(\mathfrak{p}_0,\beta_0)}\bigr)(x)=-\sum_i \lambda(z_{i})^2 x,
\]
so $\sum_i\lambda(z_i)^2=0$. Choosing $\{z_i\}\subseteq\fp_0^{\sigma}$, all $\lambda(z_i)$ are real; hence $\lambda(z_i)=0$ for every $i$ and $\lambda\equiv 0$.

Finally, since $\lambda\equiv0$ we have $K^\lambda=K$, and evaluation at the identity gives an isomorphism $V(0,\tau)\cong E^\tau$ under which $\fp_0$ acts trivially. In particular $V$ is finite-dimensional.
\end{proof}

\subsection{Dirac cohomology beyond the unitary dual in real rank one}\label{33}
Recall that a real reductive group $G(\R)$ is said to be of \textit{real rank one} if the dimension of a maximal abelian subspace $\fa^{\sigma}$ of $\fp^{\sigma}$ is one. 
\begin{theorem*}
Let $G(\R)$ be a connected real reductive group of  real rank one, and let $V$ be an irreducible $(\fg_0,K)$-module. Then $V$ has
nonzero Dirac cohomology if and only if $\fp_0$ acts trivially on it. In this case, $V$ is finite-dimensional. 
\end{theorem*}

In particular, in the real rank one case there are no irreducible $(\fg_0,K)$-modules with nonzero Dirac cohomology that are not infinitesimally unitary.

Before proving the theorem, we establish a technical lemma.

\begin{lemma*}
 Let $(\lambda,\tau)$ be a Mackey datum for a Cartan motion group $G(\R)_0$ of a connected real reductive group $G(\R)$. Assume that the $(\fg_0,K)$-module $V(\lambda,\tau)$ has nonzero Dirac cohomology and that $\operatorname{dim}_{\R}(\lambda(\fp_0^{\sigma}))\leq 1$. Then $\lambda\equiv 0$.
\end{lemma*}
\begin{proof}
We set $m:=\operatorname{dim}_{\R}(\fp_0^{\sigma})$.  Assume, by contradiction, that $\lambda\not\equiv 0$. Since $\fp_0^{\sigma}$ spans $\fp_0$ over $\C$, the restriction $\lambda|_{\fp_0^{\sigma}}$ is nonzero, so $\operatorname{dim}_{\R}(\lambda(\fp_0^{\sigma}))\geq 1$ and hence, by assumption, $\operatorname{dim}_{\R}(\lambda(\fp_0^{\sigma}))= 1$. Then the kernel of the $\R$-linear map 
\[\lambda|_{\fp_0^{\sigma}}:\fp_0^{\sigma}\longrightarrow \C\]
is of dimension $m-1$ and the orthogonal complement $\left(\operatorname{ker}(\lambda|_{\fp_0^{\sigma}})\right)^{\perp}$ inside $\fp_0^{\sigma}$ is one-dimensional.  Hence we can find a $\beta_0$-orthonormal basis $\{e_1,...,e_{m}\}$ such that $\{e_1\}$ is a basis for $\left(\operatorname{ker}(\lambda|_{\fp_0^{\sigma}})\right)^{\perp}$ and $\{e_2,...,e_{m}\}$ a basis for $\operatorname{ker}(\lambda|_{\fp_0^{\sigma}})$.
As before, $\Omega_0=\sum_{i=1}^me_i^2$ and, as in the proof of Theorem \ref{DCCMG}, we can show that \begin{equation*}
	\sum_{i=1}^m \lambda(e_{i})^2=0.
\end{equation*}
Since $\{e_2,...,e_{m}\}$ is a basis for $\operatorname{ker}(\lambda|_{\fp_0^{\sigma}})$ we obtain 
\begin{equation*}
\lambda(e_{1})^2=0.
\end{equation*}
Hence $\lambda(e_{1})=0$, and therefore $\lambda\equiv 0$, which is a contradiction. 
\end{proof}

\begin{proof}[Proof of Theorem \ref{33}]
If $\fp_0$ acts trivially on $V$, then $D_0(V)=0$ and hence
$H_{D_0}(V)=V\otimes_{\C}S_0\neq0$.

Conversely, suppose that $H_{D_0}(V)\neq 0$. By Theorem \ref{ad} we may assume that $V=V(\lambda,\tau)$ for some Mackey datum $(\lambda,\tau)$. Since $\lambda$ vanishes on $\fq$,
\[
\operatorname{dim}_{\R}(\lambda(\fp_0^{\sigma}))=\operatorname{dim}_{\R}(\lambda(\fa^{\sigma}))\leq \operatorname{dim}_{\R}(\fa^{\sigma})=1,
\]
so the preceding lemma gives $\lambda\equiv 0$. As in the proof of Theorem \ref{DCCMG}, this means that $\fp_0$ acts trivially on $V\cong E^{\tau}$, which is finite-dimensional.
\end{proof}

\section{The Cartan motion group as a special fiber   }\label{thedeformationfamily}

In this section   we promote the idea that a Cartan motion group   of a real reductive group  should be thought of as a special fiber of a family. This approach has been used substantially before, especially in the context of the Mackey--Higson bijection; see e.g. \cite{MR409726,MR2391803,MR2815133,MR4079418,MR4400734,eyallietheory,MR4542720}. 

Here we shall work with the deformation family of Harish-Chandra pairs associated with a real reductive group $G(\R)$. We will  explain how the Dirac theory for the Cartan motion group $G(\R)_0$, which  was developed in Sections  \ref{DiracTheory} and \ref{DS} intrinsically within $G(\R)_0$, can also be obtained from a special fiber of a family. 

\subsection{The deformation family}
In this subsection we recall the deformation family of Harish-Chandra pairs that is associated with  a real reductive group. For further details about algebraic families   of Harish-Chandra pairs and their modules see \cite{eyalbern,MR4123111}.
In this paper we shall only consider families in which the base variety is the affine line $\mathbb{A}^1_{\C}$. In this case one can replace a family by its space of global sections. This will be the approach we shall take throughout the paper without further notice. 

\subsubsection{Families of Lie algebras}
An \textit{algebraic family of complex Lie algebras $\boldsymbol{\mathfrak{g}}$ over $\mathbb{A}^1_{\C}$} 
is a Lie algebra over $\C[z]$ (the polynomial ring here is considered as the space of global sections of the sheaf of regular functions of  $\mathbb{A}^1_{\C}$) that is free as a $\C[z]$-module.  Since for us the base variety is always $\mathbb{A}^1_{\C}$, we shall occasionally omit it and simply say an algebraic family of complex Lie algebras.

For every $\alpha\in \C$, the fiber of 
$\boldsymbol{\mathfrak{g}}$ over $\alpha$ is the complex Lie algebra
\[\boldsymbol{\mathfrak{g}}|_{\alpha}:=\boldsymbol{\mathfrak{g}}/(I_{\alpha}\boldsymbol{\mathfrak{g}}), \]
where $I_{\alpha}$ is the maximal ideal of $\C[z]$ generated by $z-\alpha$.  

\begin{remark*}
    Recall that the fiber $\boldsymbol{\mathfrak{g}}|_{\alpha}$ is canonically isomorphic to the quotient of $\g_{\alpha}$,  the stalk of $\g$ at $\alpha$,  by $I_{\alpha}\g_{\alpha}$.
\end{remark*}
\begin{example*}
    Let $\fg$ be a complex Lie algebra.  The constant algebraic family of complex Lie algebras with fiber $\fg$ is
    \[\C[z]\otimes_{\C}\fg.\]
    Clearly every fiber of $\C[z]\otimes_{\C}\fg$ is canonically isomorphic to $\fg$. Explicitly, the evaluation at $\alpha$ isomorphism $\operatorname{ev}_{\alpha}:(\C[z]\otimes_{\C}\fg)|_{\alpha}\longrightarrow \fg$ is given by 
    \[\sum_{i}f_i(z)\otimes X_i+I_{\alpha}\C[z]\otimes_{\C}\fg\longmapsto \sum_{i}f_i(\alpha)X_i.\]
\end{example*}

\subsubsection{Families of Harish-Chandra pairs}
In this paper we only consider algebraic families of Harish-Chandra pairs with a constant group scheme. The definition below is the general definition from \cite{eyalbern,MR4123111}, adapted to the present setting of families over $\mathbb{A}^1_{\C}$ and with a constant group scheme.  

An \textit{algebraic family of Harish-Chandra pairs over $\mathbb{A}^1_{\C}$} is a pair $(\boldsymbol{\mathfrak{g}},K)$ with $\boldsymbol{\mathfrak{g}}$ being an algebraic family of complex Lie algebras over $\mathbb{A}^1_{\C}$ and $K$ being a complex algebraic group acting on $\boldsymbol{\mathfrak{g}}$ by automorphisms of Lie algebras over $\C[z]$ 
\begin{equation*}\pi_{K,\g}:K\longrightarrow \operatorname{Aut}_{\C[z]}(\boldsymbol{\mathfrak{g}}),
\end{equation*}
together with a $K$-equivariant embedding of complex Lie algebras \begin{equation*}j:\operatorname{Lie}(K)\longrightarrow \boldsymbol{\mathfrak{g}},
\end{equation*}
from  the Lie algebra of $K$ into  $\boldsymbol{\mathfrak{g}}$. Here the action of $K$ on its Lie algebra is the adjoint action. It is further required that the two  actions   of $\operatorname{Lie}(K)$ on $\boldsymbol{\mathfrak{g}}$ coincide in the sense that
\begin{equation*}\operatorname{ad}_{j(Y)}(Z)=\pi_{\fk,\g}(Y)(Z), \quad \forall Y\in \operatorname{Lie}(K),Z\in \boldsymbol{\mathfrak{g}}.
\end{equation*}
Here $\pi_{\fk,\g}$ stands for the action of $\operatorname{Lie}(K)$ that is obtained from the action $\pi_{K,\g}$ of $K$ on $\g$. 

\begin{example*}
    Let $(\fg,K)$ be a complex Harish-Chandra pair.  The constant algebraic family of Harish-Chandra pairs  with fiber $(\fg,K)$ is the pair 
    \[(\C[z]\otimes_{\C}\fg,K),\]
    where the action of $K$ on $\C[z]\otimes_{\C}\fg$ is obtained from the action on $\fg$ via scalar extension, and the embedding $\operatorname{Lie}(K)\longrightarrow \C[z]\otimes_{\C}\fg$ is obtained from   the embedding $\operatorname{Lie}(K)\longrightarrow  \fg$ composed with 
    \begin{eqnarray}\nonumber
        && \fg \longrightarrow \C[z]\otimes_{\C}\fg\\ \nonumber
        && X\longmapsto 1\otimes X.
    \end{eqnarray}

    Clearly for  every $\alpha\in \C$, the pair $((\C[z]\otimes_{\C}\fg)|_{\alpha},K)$ is a Harish-Chandra pair canonically isomorphic  to $(\fg,K)$. 
\end{example*}

\subsubsection{The deformation family}\label{df}
We now go back to the notation and assumptions of Subsection \ref{Gr}.
In particular, $G(\R)$ is a connected real reductive group, a real form of   a complex connected reductive algebraic group $G$, $K$ the fixed point subgroup of an  involution $\theta$ of $G$, and $\fg$ the complex Lie algebra of $G$ having a Cartan decomposition 
\[\mathfrak{g}=\mathfrak{k}\oplus \mathfrak{p}.\]
The pair $(\fg,K)$ is the Harish-Chandra pair of $G(\R)$. 
By definition \textit{the deformation family of $G(\R)$} is the algebraic subfamily $(\g_d,K)$ of the constant family of Harish-Chandra pairs  $(\C[z]\otimes_{\C}\fg,K)$, with 
\[\g_d:=\k\oplus \p_d, \qquad \k:=\C[z]\otimes_{\C}\fk,\quad    \p_d:= z\C[z]\otimes_{\C}\fp.\]

The involution $\theta$ of $\fg$ can be extended to an involution $\boldsymbol{\theta}$ of $\g_d$ whose eigenspaces are $\k$ and $\p_d$. 

One can  verify that if $\{X_1,...,X_d\}$ is a basis for $\fk$ and $\{X_{d+1},...,X_{n}\}$ is  a basis for $\fp$ with structure constants $\{C_{ij}^k\}$, that is, with  $[X_i,X_j]=\sum_{k=1}^nC_{ij}^k{X}_k$,  then the set $\{\boldsymbol{X}_1:=1\otimes X_1,...,\boldsymbol{X}_d:=1\otimes X_d\}$ is a basis for $\k$ over $\C[z]$, and $\{\boldsymbol{X}_{d+1}:=z\otimes X_{d+1},...,\boldsymbol{X}_n:=z\otimes X_{n}\}$ is a basis for $\p_d$ over $\C[z]$. 

Moreover, the structure constants of the basis $\{\boldsymbol{X}_1,...,\boldsymbol{X}_n\}$ are given by \[[\boldsymbol{X}_i,\boldsymbol{X}_j]=\begin{cases}
      \sum_{k=1}^nC_{ij}^k\boldsymbol{X}_k, &   i\leq d \hspace{1mm}\text{or} \hspace{1mm} j \leq d \\
     \sum_{k=1}^dC_{ij}^k z^2 \boldsymbol{X}_k  ,&  d+1\leq i,j\leq n.
    \end{cases} \]

\begin{lemma*}
    The unique complex linear map $\eta:\g_d|_0\longrightarrow \fg_0$ 
    satisfying \begin{eqnarray}\nonumber
      &&\eta(f(z)\otimes X+I_0\g_d)=\begin{cases}
        f(0)X,& X\in \fk\\
        f'(0)X,& X\in \fp
        \end{cases}
\end{eqnarray} 
for every $f(z)\otimes X$ that is an eigensection of $\boldsymbol{\theta}$, 
is an isomorphism of complex Lie algebras. 
\end{lemma*}
 \begin{proof}
By direct calculation $\eta$ is well-defined.
 Fix bases for $\fk$ and $\fp$ and use the notation of the discussion before the Lemma. Note that 
$\{\boldsymbol{X}_1+I_0\g_d,...,\boldsymbol{X}_n+I_0\g_d\}$ is a basis for $\g_d|_0$ and $\eta$ satisfies 
\begin{eqnarray}\nonumber
    &&\eta\left(\sum_{i=1}^nf_i(z) (\boldsymbol{X}_i+I_0\g_d)\right)=\sum_{i=1}^nf_i(0)X_i.
\end{eqnarray}

Since $\eta$ takes the basis $\{\boldsymbol{X}_1+I_0\g_d,...,\boldsymbol{X}_n+I_0\g_d\}$ of  $\g_d|_0$ to the basis $\{X_1,...,X_n\}$ of $\fg_0$, $\eta$ is a linear bijection. To see that $\eta$ is a morphism of Lie algebras observe that 
\begin{eqnarray}\nonumber
&& \eta \left([\boldsymbol{X}_i+I_0\g_d,\boldsymbol{X}_j+I_0\g_d] \right)=\eta   \left([\boldsymbol{X}_i,\boldsymbol{X}_j] +I_0\g_d \right)\\ \nonumber
&&=\begin{cases}
      \eta \left(\sum_{k=1}^nC_{ij}^k\boldsymbol{X}_k+I_0\g_d \right), &   i\leq d \hspace{1mm}\text{or} \hspace{1mm} j \leq d \\
     \eta \left(\sum_{k=1}^dC_{ij}^k z^2 \boldsymbol{X}_k+I_0\g_d \right)  ,&  d+1\leq i,j\leq n
     \end{cases}\\ \nonumber
    && =\begin{cases}
    \sum_{k=1}^nC_{ij}^k{X}_k , &   i\leq d \hspace{1mm}\text{or} \hspace{1mm} j \leq d \\
     0  ,&  d+1\leq i,j\leq n
     \end{cases}\\ \nonumber
    && =[X_i,X_j]_0=[\eta \left(\boldsymbol{X}_i+I_0\g_d\right),\eta \left(\boldsymbol{X}_j+I_0\g_d\right)]_0. 
\end{eqnarray}
 \end{proof}

\subsubsection{The zero fiber of the universal enveloping algebra}\label{414}
We denote the universal enveloping algebra of $\g_d$ by $\mathcal{U}(\g_d)$. It is an associative $\C[z]$-algebra together with a morphism of Lie algebras over $\C[z]$, $\iota:\g_d\longrightarrow \mathcal{U}(\g_d)$ such that for any associative $\C[z]$-algebra $A$ and a morphism of Lie algebras  over $\C[z]$, $\phi:\g_d\longrightarrow A$ there exists a unique morphism of $\C[z]$-algebras $\widetilde{\phi}:\mathcal{U}(\g_d)\longrightarrow A$ such that $\phi=\widetilde{\phi}\circ \iota$. 
The universal enveloping algebra is unique up to isomorphism. We shall realize $\mathcal{U}(\g_d)$ as the quotient of the tensor algebra $T(\g_d)$ by the two-sided ideal generated by 
\[X\otimes Y-Y\otimes X-[X,Y], \quad \forall X,Y\in \g_d.\]

\begin{lemma*}
   Let $\psi:\g_d|_0\longrightarrow \mathcal{U}(\g_d)|_0$ be the natural morphism of complex Lie algebras that is defined by 
  \[\psi(X+I_0\g_d)=\iota(X)+I_0\mathcal{U}(\g_d), \quad \forall X\in \g_d.\]
  Then the induced map from the universal property of the enveloping algebra $\widetilde{\psi}:\mathcal{U}(\g_d|_0)\longrightarrow \mathcal{U}(\g_d)|_0$ is an isomorphism.
\end{lemma*}

The lemma immediately follows from base change properties for enveloping algebras over commutative rings, see 
\cite[Ch.~I, \S 2, no.~9, p.~25]{BourbakiLieGroupsLieAlgebrasCh1to3}.
We denote by $\widetilde{\eta}:\mathcal{U}(\g_d|_0)\longrightarrow \mathcal{U}(\fg_0)$ the isomorphism of algebras induced by the Lie algebra isomorphism $\eta$ of Subsubsection \ref{df}.

\begin{corollary*}
    The map 
    $$\widetilde{\eta}\circ \widetilde{\psi}^{-1}:\mathcal{U}(\g_d)|_0
\longrightarrow \mathcal{U}(\fg_0)$$
   is an isomorphism.
\end{corollary*}

Explicitly for 
$f(z)\otimes X\in \k\cup \p_d\subset \g_d$, 
\[
\left(\widetilde{\eta}\circ\widetilde{\psi}^{-1}\right)
\left(f(z)\otimes X+I_0\mathcal U(\g_d)\right)
=
\begin{cases}
f(0)X,  & X\in\fk,\\
f'(0)X, & X\in\fp.
\end{cases}
\]
 
\subsection{The Clifford algebra of the deformation family}
In this subsection we recall the construction of the Clifford algebra for the deformation family of  a real reductive group. This is a special case of the theory that is covered in \cite{afentoulidisalmpanis2025diracoperatorsalgebraicfamilies}. We shall only describe the setup and quote the relevant results. For proofs and further details see the above mentioned reference.

\subsubsection{The orthogonalizable form on $\p_d$}
We keep the symmetric bilinear form $\beta$ on $\fg$ fixed as in Subsubsection \ref{invform}; recall that it satisfies conditions (1)--(4) there. In particular, $\beta$ is nondegenerate, $\theta$-invariant, and $(\fg,K)$-invariant.

We denote by  $1\otimes {\beta}$ the symmetric bilinear form  on $\C[z]\otimes_{\mathbb{C}}\mathfrak{g}$ obtained by $\C[z]$-linearity from the form $\beta$.

\begin{lemma*}[See Lemma 4.2.2 of    
\cite{afentoulidisalmpanis2025diracoperatorsalgebraicfamilies}]
    There is a unique $\C[z]$-bilinear symmetric
     $(\boldsymbol{\fk},K)$-invariant   and orthogonalizable form $ \boldsymbol{\beta}_{d}:\boldsymbol{\fp}_d\times \boldsymbol{\fp}_d\longrightarrow  \C[z]$ such that 
    \[ \left(1\otimes{\beta}\right)|_{\boldsymbol{\fp}_d\times \boldsymbol{\fp}_d}=z^2\boldsymbol{\beta}_{d}.\]
\end{lemma*}
Recall that orthogonalizable means that there is a basis for $\p_d$ in which the form is represented by the identity matrix.  In particular, the form is unimodular, that is,  the induced morphism $\p_d\longrightarrow \p_d^*=\operatorname{Hom}_{\C[z]}(\p_d,\C[z])$ is an isomorphism of $\C[z]$-modules. 

We identify the image  of $\p_d$ in $\g_d|_0$ under the canonical projection with $\p_d|_0$.
The $\C[z]$-bilinear form $\boldsymbol{\beta}_d$ induces a $\C$-bilinear form $\boldsymbol{\beta}_{d}|_0$ on $\p_d|_0$ given by 
\[\boldsymbol{\beta}_{d}|_0(\boldsymbol{X}+I_0\p_d,\boldsymbol{Y}+I_0\p_d):=\boldsymbol{\beta}_{d}(\boldsymbol{X},\boldsymbol{Y})(0). \]

\begin{proposition*}
The map $\eta$ intertwines the form $\boldsymbol{\beta}_{d}|_0$ on $\p_d|_0$ and the form $\beta_0$ on $\fp_0$. 
Explicitly,  
the following diagram is commutative

\[\xymatrix{
& \p_d|_0\otimes_{\C} \p_d|_0\ar[d]^{\eta\otimes \eta } \ar[r]^{\quad \boldsymbol{\beta}_{d}|_0}
& \C   \\
&\fp_0\otimes_{\C}\fp_0\ar[ru]_{\beta_0}  &   } \]
\end{proposition*}
 \begin{proof}
 We use the bases convention and notation of Subsubsection \ref{df}. 
For every $d< i,j\leq n$ 
\begin{eqnarray}\nonumber
&&   \boldsymbol{\beta}_{d}|_0(\boldsymbol{X}_i+I_0\p_d,\boldsymbol{X}_j+I_0\p_d)=\boldsymbol{\beta}_{d}(\boldsymbol{X}_i,\boldsymbol{X}_j)(0)=\boldsymbol{\beta}_{d}(z\otimes {X}_i,z\otimes{X}_j)(0)=\\ \nonumber
&&\left(z^{-2}(1\otimes \beta)(z\otimes X_i,z\otimes X_j) \right)(0)=\beta_0(X_i,X_j)=\beta_0(\eta(\boldsymbol{X}_i+I_0\p_d),\eta(\boldsymbol{X}_j+I_0\p_d)).
\end{eqnarray}    
 \end{proof}   

\subsubsection{The quadratic spaces }\label{422}
By abuse of terminology we shall refer to a $\C$-module with a $\C$-bilinear symmetric form, respectively a $\C[z]$-module with a $\C[z]$-bilinear symmetric form, as a quadratic space over $\C$, respectively over $\C[z]$. 

We define a canonical complex subspace  $\fp_1:=\C z\otimes_{\C}\fp$ of $\p_d$.
 Proposition~4.3.1 of \cite{afentoulidisalmpanis2025diracoperatorsalgebraicfamilies} implies the following lemma. 

\begin{lemma*} There is a canonical $( \fk,K)$-equivariant isomorphism $T_1$ of quadratic spaces from  
$(\fp,\beta)$
onto 
$( \fp_1,\boldsymbol{\beta}_d|_{\fp_1})$. Explicitly, it  is given by $T_1(X)=z\otimes X$ for every $X\in \fp$.
In addition, the $\C[z]$-linear extension of $T_1$ to a morphism $T: \C[z]\otimes_{\C}\fp \longrightarrow \p_d$ is an isomorphism between the quadratic spaces $(\C[z]\otimes_{\C}\fp,1\otimes \beta)$ and $(\p_d,\boldsymbol{\beta}_d)$.
\end{lemma*}

In particular it follows from the above Lemma that   the induced map between the fibers  $ T|_0:\left(\C[z]\otimes_{\C}\fp\right)|_{0} \longrightarrow \p_d|_{0}$, which is explicitly  given by 
\[f(z)\otimes X +I_{0}\left(\C[z]\otimes_{\C}\fp\right)\longmapsto f(z)T_1(X) +I_0\p_d\]
is an isomorphism of quadratic spaces.  

We note that the form $(1\otimes \beta)|_{0}$ on $\left(\C[z]\otimes_{\C}\fp\right)|_{0}$ corresponds to the form $\beta$ under the canonical evaluation  isomorphism $\operatorname{ev}_0:\left(\C[z]\otimes_{\C}\fp\right)|_{0}\longrightarrow \fp$.  
Hence  we obtain an isomorphism of quadratic spaces \[\operatorname{ev}_0\circ T|_0^{-1}:(\p_d|_0,\boldsymbol{\beta}_d|_0)\longrightarrow (\fp,\beta).\] 
We observe that $\operatorname{ev}_0\circ T|_0^{-1}$ coincides with the restriction of $\eta$ to ${\p_d|_0}$.

\subsubsection{The Clifford algebra of the deformation family }\label{423}

We denote by $\boldsymbol{\operatorname{Cl}}_d$  or  by $\operatorname{Cl}(\p_d,\boldsymbol{\beta}_{d})$ the Clifford algebra of the quadratic space $(\p_d,\boldsymbol{\beta}_{d})$ realized  as  the quotient of the tensor algebra $T(\p_d)$ by the two-sided ideal $I(\p_d,\boldsymbol{\beta}_d)$ generated by all elements of the form
\begin{equation*}
	X\otimes Y+Y\otimes X-\boldsymbol{\beta}_d (X,Y), \quad X,Y\in\p_d.
\end{equation*} 

We denote   the canonical embedding of $\p_d$ into $\operatorname{Cl}(\p_d,\boldsymbol{\beta}_{d})$ by $\boldsymbol{\gamma}_d$.

The isomorphism $T: (\C[z]\otimes_{\C}\fp,1\otimes \beta) \longrightarrow (\p_d,\boldsymbol{\beta}_d)$ of  quadratic spaces from Subsubsection \ref{422}  induces an isomorphism of Clifford algebras \[T:\operatorname{Cl}(\C[z]\otimes_{\C}\fp,1\otimes \beta) \longrightarrow \operatorname{Cl}(\p_d,\boldsymbol{\beta}_d).\]

Similarly the isomorphisms $T|_0$ and $\operatorname{ev}_0\circ T|_0^{-1}$ between the quadratic spaces in Subsubsection \ref{422}  
 induce  isomorphisms between the corresponding Clifford algebras. Abusing notation,  we also  denote them  by  
\[T|_0:\operatorname{Cl}\left(\left(\C[z]\otimes_{\C}\fp\right)|_{0}, (1\otimes \beta)|_{0}\right) \longrightarrow \operatorname{Cl}\left(\p_d|_{0},\boldsymbol{\beta}_d|_0 \right),\]
and by
 \[\operatorname{ev}_0\circ T|_0^{-1}:\operatorname{Cl}(\p_d|_0,\boldsymbol{\beta}_d|_0)\longrightarrow \operatorname{Cl}(\fp,\beta).\]

\begin{lemma*}
 There is a  canonical isomorphism 
 \[\tau:\operatorname{Cl}(\p_d|_0,\boldsymbol{\beta}_d|_0)\longrightarrow  \operatorname{Cl}(\p_d,\boldsymbol{\beta}_d)|_0,\]
 satisfying 
 \[\tau(X+I_0\p_d)=\boldsymbol{\gamma}_d(X)+I_0\operatorname{Cl}(\p_d,\boldsymbol{\beta}_d), \quad \forall X\in \p_d.\]
\end{lemma*}

\begin{proof}
    This follows from the base-change property
of Clifford algebras \cite[Algèbre, Ch.~9, \S 9, no.~1, Prop.~2]{BourbakiAlgebre9}. 
\end{proof}

\begin{corollary*}
The map   \[\operatorname{ev}_0\circ T|_0^{-1}\circ \tau^{-1}:  \operatorname{Cl}(\p_d,\boldsymbol{\beta}_d)|_0\longrightarrow  \operatorname{Cl}(\fp,{\beta}),\]
is an isomorphism. 
\end{corollary*}

\subsubsection{The embeddings of $\fk$ in $\operatorname{Cl}(\p_d,\boldsymbol{\beta}_d)$ }\label{s424}

Similarly to Subsubsection \ref{223}, and following \cite[Sec.~4.6.1]{afentoulidisalmpanis2025diracoperatorsalgebraicfamilies}, we define a morphism of Lie algebras 
over  $\C[z]$ by 
\begin{eqnarray}\nonumber
   && \varphi_{d}:\mathfrak{so}(\boldsymbol{\fp}_d,\boldsymbol{\beta}_d )\longrightarrow \operatorname{Cl}(\boldsymbol{\fp}_d,\boldsymbol{\beta}_d )\\ \nonumber
   && \varphi_d(R_{\boldsymbol{\beta}_d,a,b})=\frac{1}{2}[\boldsymbol{\gamma}_d(a),\boldsymbol{\gamma}_d(b)],
\end{eqnarray}
where for $a,b\in \p_d$, $R_{\boldsymbol{\beta}_d,a,b} $ is the operator in $\mathfrak{so}(\boldsymbol{\fp}_d,\boldsymbol{\beta}_d )$ that is given by 
\[R_{\boldsymbol{\beta}_d,a,b}(X)=\boldsymbol{\beta}_d(X,b)a-\boldsymbol{\beta}_d(X,a)b. \]

The adjoint action of $\k$ on $\p_d$ defines a morphism
\[
\operatorname{ad}_d:\k\longrightarrow
\mathfrak{so}(\p_d,\boldsymbol\beta_d),
\qquad
\operatorname{ad}_d(Y)(X)=[Y,X].
\]
We put $\alpha_d:=\varphi_d\circ\operatorname{ad}_d:\k \longrightarrow  \operatorname{Cl}(\boldsymbol{\fp}_d,\boldsymbol{\beta}_d )$. It is  a morphism of Lie algebras over $\C[z]$.

\subsection{The generalized pair of the deformation family}
In this subsection we recall the definition of the generalized pairs associated with the deformation family. 
 
An \textit{algebraic family of complex (associative) algebras (over  $\mathbb{A}^1_{\C}$)}
is a free $\C[z]$-module  that is also an algebra over $\C[z]$.

Following the notion of a \textit{pair} (or a \textit{generalized pair}) of \cite[Ch.~I.6]{KNV}, see also \cite[Sec.~3.1.3]{afentoulidisalmpanis2025diracoperatorsalgebraicfamilies}, an \textit{algebraic family of generalized (Harish-Chandra) pairs} over $\mathbb{A}^1_{\C}$ is a pair $(\boldsymbol{A},K)$, with $\boldsymbol{A}$ an algebraic family of complex algebras over $\mathbb{A}^1_{\C}$ and $K$ a complex algebraic group acting on $\boldsymbol{A}$ by $\C[z]$-algebra automorphisms, together with a $K$-equivariant morphism of complex Lie algebras $\Delta:\operatorname{Lie}(K)\longrightarrow \boldsymbol{A}$ such that the differential of the action of $K$ on $\boldsymbol{A}$ is given by the inner derivations $Y\longmapsto [\Delta(Y),\_\,]$, $Y\in\operatorname{Lie}(K)$.

\subsubsection{The generalized pair $(\boldsymbol{A}(\g_d,\boldsymbol{\beta}_{d}),\widetilde{K})$}\label{s432}
We define an associative $\C[z]$-algebra via
\[\boldsymbol{A}_d=\boldsymbol{A}(\g_d,\boldsymbol{\beta}_{d}):=\mathcal{U}(\g_d)\otimes_{\C[z]}\operatorname{Cl}(\p_d,\boldsymbol{\beta}_{d}).\]
The group $K$ acts on $\mathcal U(\g_d)$ and
$\operatorname{Cl}(\p_d,\boldsymbol\beta_d)$ by
$\C[z]$-algebra automorphisms. Hence it acts diagonally on
$\boldsymbol A_d$ by $\C[z]$-algebra automorphisms.
The map
\begin{eqnarray}\nonumber
  && \Delta_d:\boldsymbol{\fk}\longrightarrow\boldsymbol{A}(\g_d,\boldsymbol{\beta}_{d}) \\ \nonumber
  &&\Delta_d(X):=X\otimes 1_{\operatorname{Cl}(\p_d,\boldsymbol{\beta}_{d})}+1_{\mathcal{U}(\g_d)}\otimes \alpha_{d}(X)
  \end{eqnarray}
is $K$-equivariant.
The pair $(\boldsymbol{A}(\g_d,\boldsymbol{\beta}_{d}),K)$ together with $\Delta_d$ form an algebraic family of generalized pairs over $\mathbb{A}^1_{\C}$.

Using the canonical homomorphism  $\mathcal{C}:\widetilde{K}\longrightarrow K$ the spin double cover $\widetilde{K}$ acts on $\boldsymbol{A}(\g_d,\boldsymbol{\beta}_{d})$ and $(\boldsymbol{A}(\g_d,\boldsymbol{\beta}_{d}),\widetilde{K})$
is an algebraic family of generalized pairs over $\mathbb{A}^1_{\C}$.
\begin{lemma*}
Using the canonical tensor-product identification $\boldsymbol{A}(\g_d,\boldsymbol{\beta}_{d})|_0\cong \mathcal{U}(\g_d)|_0\otimes_{\C}\operatorname{Cl}(\p_d,\boldsymbol{\beta}_{d})|_0$,
put
\[
\Phi_U:=
\widetilde\eta\circ\widetilde\psi^{-1},
\qquad
\Phi_{\mathrm{Cl}}
:=
\operatorname{ev}_0\circ T|_0^{-1}\circ\tau^{-1}.
\]
The tensor product
\[
\Phi_A
:=
\Phi_U\otimes_{\C}\Phi_{\mathrm{Cl}}
:
\boldsymbol A_d|_0
\longrightarrow
A_0
=
\mathcal U(\fg_0)\otimes_{\C}
\operatorname{Cl}(\fp_0,\beta_0)
\]
is a $\widetilde K$-equivariant algebra isomorphism.
\end{lemma*}
This follows from Corollary \ref{414} and Corollary \ref{423}.

\subsection{Canonical sections, Casimir, and Dirac operator}\label{s44}
In this subsection we recall the definitions of the canonical  Casimir and the Dirac operator for the deformation family. 
\subsubsection{Canonical sections}
 Recall (see \cite[Sec.~4.4]{afentoulidisalmpanis2025diracoperatorsalgebraicfamilies}) that on any $\C[z]$-submodule $\boldsymbol{\l}$ of $\p_d$ on which the restriction of $\boldsymbol{\beta}_d$ is unimodular we have the canonical isomorphisms 
\begin{equation}\label{9} 
	\mathrm{End}_{\C[z]}(\boldsymbol{\mathfrak{l}})\longrightarrow  \boldsymbol{\mathfrak{l}}^*\otimes_{\C[z]}\boldsymbol{\mathfrak{l}}\longrightarrow \boldsymbol{\mathfrak{l}}\otimes_{\C[z]}\boldsymbol{\mathfrak{l}}.
\end{equation}
There is a  canonical element 
$\omega(\boldsymbol{\mathfrak{l}},\boldsymbol{\beta}_d)$ in $\boldsymbol{\mathfrak{l}}\otimes_{\C[z]}\boldsymbol{\mathfrak{l}}$ that is defined to be the image under the above-mentioned isomorphisms \eqref{9} of  the identity operator $\operatorname{Id}_{\boldsymbol{\mathfrak{l}}}\in \mathrm{End}_{\C[z]}(\boldsymbol{\mathfrak{l}})$.

We denote the image of $\omega(\boldsymbol{\mathfrak{l}},\boldsymbol{\beta}_d)$ in $\mathcal{U}(\g_d)$
under  the obvious morphisms 
  \[\boldsymbol{\mathfrak{l}}\otimes_{\C[z]}\boldsymbol{\mathfrak{l}}\longrightarrow T(\boldsymbol{\mathfrak{l}})\longrightarrow T(\g_d)\longrightarrow \mathcal{U}(\g_d),\]   by $\Omega(\boldsymbol{\mathfrak{l}},\boldsymbol{\beta}_d)$.

  \subsubsection{The Casimir of the deformation family}

Let $\Omega(\fg,\beta)$ be the Casimir element of the reductive
Lie algebra $\fg$, see Subsubsection \ref{s243}. Following
\cite[Sec.~4.4]{afentoulidisalmpanis2025diracoperatorsalgebraicfamilies},
define the Casimir element of $\g_d$ by
\[
\Omega_d
:=\Omega(\g_d,\beta,z)
:=z^2\otimes\Omega(\fg,\beta)
\in\mathcal U(\g_d)
\subseteq\C[z]\otimes_{\C}\mathcal U(\fg).
\]
This element is central in $\mathcal U(\g_d)$. Moreover,
\[
\Omega_d
=
z^2\otimes\Omega(\fk,\beta|_{\fk})
+\Omega(\p_d,\boldsymbol\beta_d).
\]

\begin{lemma*}
Under the zero-fiber isomorphism,
\[
\left(\widetilde{\eta}\circ\widetilde{\psi}^{-1}\right)
\left(\Omega_d+I_0\mathcal U(\g_d)\right)
=
\Omega(\fp_0,\beta_0)
=
\Omega_0.
\]
\end{lemma*}

\begin{proof}
Choose a $\beta_0$-orthonormal basis
$\{X_{d+1},\ldots,X_n\}$ of $\fp_0$. Then
$\{z\otimes X_{d+1},\ldots,z\otimes X_n\}$ is a
$\boldsymbol\beta_d$-orthonormal basis of $\p_d$, and hence
\[
\Omega_d
=
z^2\otimes\Omega(\fk,\beta|_{\fk})
+\sum_{i=d+1}^{n}(z\otimes X_i)^2.
\]
The first summand vanishes in the zero fiber, while
\[
\widetilde{\eta}\circ \widetilde{\psi}^{-1}\bigl((z\otimes X_i)+I_0\mathcal U(\g_d)\bigr)=X_i.
\]
Therefore
\[
\widetilde{\eta}\circ \widetilde{\psi}^{-1}\bigl(\Omega_d+I_0\mathcal U(\g_d)\bigr)
=
\sum_{i=d+1}^{n}X_i^2
=
\Omega(\fp_0,\beta_0).
\]
\end{proof}

  \subsubsection{The Dirac operator of the deformation family} 
The \textit{Dirac operator} $D(\g_d,\boldsymbol{\beta}_{d})$ of $\g_d$ (with respect to the fixed symmetric form $\boldsymbol{\beta}_d$) is the image of the canonical element $\omega(\p_d,\boldsymbol{\beta}_d)$  in $\boldsymbol A_d
=\boldsymbol A(\g_d,\boldsymbol{\beta}_{d})$
under the tensor product of the morphism $\p_d\longrightarrow \mathcal{U}(\g_d)$ and the morphism $\boldsymbol{\gamma}_d:\p_d\longrightarrow \operatorname{Cl}(\p_d,\boldsymbol{\beta}_{d})$. 
If $\{X_{d+1},...,X_n\}$ is an orthonormal basis for $\fp_0$ with respect to $\beta_0$, then
\begin{equation*}
		D(\g_d,\boldsymbol{\beta}_{d})=\sum_{i=d+1}^n(z\otimes X_i)\otimes\boldsymbol{\gamma}_d(z\otimes X_i).
	\end{equation*}
 Recall the isomorphism 
 $
\Phi_A:
\boldsymbol A_d|_0
\longrightarrow
A_0$  from Subsubsection \ref{s432}.
\begin{lemma*}
Under the preceding isomorphism,
\[
\Phi_A\bigl(
D(\g_d,\boldsymbol\beta_d)+I_0\boldsymbol A_d
\bigr)
=
D_0.
\]
\end{lemma*}

\begin{proof}
By Subsubsection \ref{414}, $\Phi_U\bigl((z\otimes X_i)+I_0\mathcal U(\g_d)\bigr)=X_i$, and by Subsubsections \ref{422} and \ref{423},
$\Phi_{\mathrm{Cl}}\bigl(\boldsymbol{\gamma}_d(z\otimes X_i)+I_0\operatorname{Cl}(\p_d,\boldsymbol{\beta}_d)\bigr)
=\operatorname{ev}_0\circ T|_0^{-1}\bigl(z\otimes X_i+I_0\p_d\bigr)=\gamma_0(X_i)$. Hence
\[\Phi_A\bigl(D(\g_d,\boldsymbol\beta_d)+I_0\boldsymbol A_d\bigr)=\sum_{i=d+1}^n X_i\otimes\gamma_0(X_i)=D_0.\]
\end{proof}

\section{Dirac cohomology and specialization at zero}
\label{DiracCohomologyAndSpecialization}

In this section, for any algebraic family of Harish-Chandra modules $\boldsymbol{V}$ for $(\g_d,K)$,
we construct a canonical comparison morphism 
$H_{D(\g_d,\boldsymbol\beta_d)}(\boldsymbol V)|_0
\longrightarrow
H_{D_0}(\boldsymbol V|_0)$
from the zero fiber of the Dirac cohomology of $\boldsymbol{V}$,  into the 
Dirac cohomology of the $(\fg_0,K)$-module $\boldsymbol{V}|_0$. We give a sufficient condition for this morphism to be injective and
an example showing that it need not be surjective.

Throughout this section we put
\[
R:=\C[z],\qquad \C_0:=R/I_0=R/(z).
\]
We identify $\C_0$ with $\C$ by evaluation at zero.  At the beginning
of Section \ref{thedeformationfamily}, the fiber of an algebraic
family at zero was defined as a quotient by $I_0$.  For any
$R$-module $\boldsymbol M$, this quotient is canonically isomorphic to
 an extension of scalars via 
\[
\begin{split}
\chi_{\boldsymbol M}:\C_0\otimes_R\boldsymbol M
&\overset{\simeq}{\longrightarrow}
\boldsymbol M/I_0\boldsymbol M,\\
(f+I_0)\otimes m&\longmapsto fm+I_0\boldsymbol M.
\end{split}
\]
Abusing notation we shall use the  scalar-extension realization and simply write
\[
\boldsymbol M|_0:=\C_0\otimes_R\boldsymbol M.
\] 
The canonical map to the zero fiber is
therefore
\[
q_{0,\boldsymbol M}:\boldsymbol M\longrightarrow\boldsymbol M|_0,
\qquad m\longmapsto1\otimes m;
\]
under $\chi_{\boldsymbol M}$, it is the usual quotient map.  If
$F:\boldsymbol M\longrightarrow\boldsymbol N$ is $R$-linear, we write
\[
F|_0:=\mathbb I_{\C_0}\otimes F:
\boldsymbol M|_0\longrightarrow\boldsymbol N|_0
\]
for the corresponding map between the fibers.
\subsection{Dirac cohomology for the deformation family}\label{s51}
In this subsection we recall the definitions and constructions needed  for  the Dirac cohomology of an algebraic family of Harish-Chandra modules
for  the deformation family $(\g_d,K)$, as described in \cite[Sec.~4]{afentoulidisalmpanis2025diracoperatorsalgebraicfamilies}.

Fix a maximal isotropic subspace $\fp_0^-$ as in Subsubsection
\ref{2.5.1} (and, when $\dim\fp_0$ is odd, the vector $e_0$ and the sign $\varepsilon$ chosen there), and put
\[
S_0:=\bigwedge\fp_0^-,\qquad
\p_d^-:=zR\otimes_{\C}\fp_0^-,\qquad
\boldsymbol{S}_d:=\bigwedge_R\p_d^-.
\]
The isometry
\[
T:R\otimes_{\C}\fp\longrightarrow\p_d
\]
from Subsubsection \ref{422} restricts to an isomorphism of $R$-modules
$R\otimes_{\C}\fp_0^-\longrightarrow\p_d^-$ and hence induces an
isomorphism
\[
T_S:R\otimes_{\C}S_0\longrightarrow\boldsymbol{S}_d.
\]
Transporting  the action of $R\otimes_{\C}\operatorname{Cl}(\fp,{\beta})$ on $R\otimes_{\C}S_0$ through
$T$ and $T_S$ equips $\boldsymbol{S}_d$ with a
$\operatorname{Cl}(\p_d,\boldsymbol{\beta}_d)$-module structure.  We
denote the corresponding representation by
\[
\boldsymbol{\gamma}'_d:
\operatorname{Cl}(\p_d,\boldsymbol{\beta}_d)
\longrightarrow\operatorname{End}_R(\boldsymbol{S}_d).
\]
This is the spin module for the deformation family constructed in
\cite[Sec.~4.8.2]{afentoulidisalmpanis2025diracoperatorsalgebraicfamilies}.
Moreover, the action of $\fk$ on $\boldsymbol{S}_d$ lifts to an action
of $\widetilde K$ by $R$-linear automorphisms, and $T_S$ is
$\widetilde K$-equivariant; see
\cite[Sec.~4.8.3]{afentoulidisalmpanis2025diracoperatorsalgebraicfamilies}.

Let $\boldsymbol V$ be an algebraic family of Harish-Chandra modules
for $(\g_d,K)$.  In particular, $\boldsymbol V$ is flat over $R$; see
\cite[Sec.~3.3.1]{afentoulidisalmpanis2025diracoperatorsalgebraicfamilies}.
Set
\[
\boldsymbol M(\boldsymbol V)
:=\boldsymbol V\otimes_R\boldsymbol S_d.
\]
It is a module for
\[
\boldsymbol A_d
=\mathcal U(\g_d)\otimes_R
\operatorname{Cl}(\p_d,\boldsymbol\beta_d)
\]
via
\[
\pi_{\boldsymbol A_d,\boldsymbol V}(u\otimes c)(v\otimes s)
=\pi_{\mathcal U(\g_d),\boldsymbol V}(u)v
 \otimes\boldsymbol\gamma'_d(c)s,
\]
for $u\in\mathcal U(\g_d)$,
$c\in\operatorname{Cl}(\p_d,\boldsymbol\beta_d)$,
$v\in\boldsymbol V$, and $s\in\boldsymbol S_d$.  Together with the
tensor product action of $\widetilde K$, this makes
$\boldsymbol M(\boldsymbol V)$ a module for the algebraic family of
generalized pairs $(\boldsymbol A_d,\widetilde K)$, as in
\cite[Sec.~4.9.1]{afentoulidisalmpanis2025diracoperatorsalgebraicfamilies}.

We shall denote the operator by which $D(\g_d,\boldsymbol\beta_d)$ acts on $\boldsymbol M(\boldsymbol V)$ by 
\[
\delta_{\boldsymbol V}
:=\pi_{\boldsymbol A_d,\boldsymbol V}
  \bigl(D(\g_d,\boldsymbol\beta_d)\bigr)
\in\operatorname{End}_R\bigl(\boldsymbol M(\boldsymbol V)\bigr).
\]
The \textit{Dirac cohomology} of $\boldsymbol{V}$ is the $R$-module 
\[
H_{D(\g_d,\boldsymbol\beta_d)}(\boldsymbol V)
:=
\frac{\ker\delta_{\boldsymbol V}}
{\ker\delta_{\boldsymbol V}\cap\operatorname{im}\delta_{\boldsymbol V}}
\]
which carries a natural action of $\widetilde K$ by automorphisms of $R$-modules.  This is the definition of \cite[Sec.~4.9.2]{afentoulidisalmpanis2025diracoperatorsalgebraicfamilies}, specialized to the deformation family.

\subsection{The specialized  Dirac operator $\delta_{\boldsymbol V}|_0$}\label{s52}

In this subsection we show that  for a $(\g_d,K)$-module $\boldsymbol{V}$, the quotient space \[\frac{\ker(\delta_{\boldsymbol V}|_0)}
{\ker(\delta_{\boldsymbol V}|_0)
 \cap\operatorname{im}(\delta_{\boldsymbol V}|_0)}\] that is obtained from the specialization of the action of the Dirac operator of the deformation family is canonically isomorphic
to the Dirac cohomology 
\[H_{D_0}(\boldsymbol V|_0).\]  

We first record explicitly how tensor products behave under
specialization.  If $\boldsymbol M$ and $\boldsymbol N$ are
$R$-modules, there is a natural isomorphism
\begin{equation}
\begin{split}
\vartheta_{\boldsymbol M,\boldsymbol N}:{}
(\boldsymbol M|_0)\otimes_{\C}(\boldsymbol N|_0)
&\overset{\simeq}{\longrightarrow}
(\boldsymbol M\otimes_R\boldsymbol N)|_0,\\
(a\otimes m)\otimes(b\otimes n)
&\longmapsto ab\otimes(m\otimes n).
\end{split}
\label{eq:tensor-base-change}
\end{equation} 
This is the usual compatibility of 
 extension-of-scalars and tensor product.

In particular, suppose that $\boldsymbol A$ is an $R$-algebra and that
\[
\pi_{\boldsymbol A,\boldsymbol M}:
\boldsymbol A\longrightarrow\operatorname{End}_R(\boldsymbol M)
\]
is a representation.  We denote the corresponding action map by
\[
\widetilde{\pi}_{\boldsymbol A,\boldsymbol M}:
\boldsymbol A\otimes_R\boldsymbol M
\longrightarrow\boldsymbol M,
\qquad
x\otimes m\longmapsto
\pi_{\boldsymbol A,\boldsymbol M}(x)(m).
\]
The specialized representation of $\boldsymbol A|_0$ on
$\boldsymbol M|_0$ has associated action map given on pure tensors by
\begin{equation}
\begin{split}
&\widetilde{\pi}_{\boldsymbol A|_0,\boldsymbol M|_0}
\bigl((a\otimes x)\otimes(b\otimes m)\bigr)\\
&\hspace{20mm}
=ab\otimes
\widetilde{\pi}_{\boldsymbol A,\boldsymbol M}(x\otimes m)
=ab\otimes
\pi_{\boldsymbol A,\boldsymbol M}(x)(m),
\end{split}
\label{eq:specialized-action}
\end{equation}
for $a,b\in\C_0$, $x\in\boldsymbol A$, and
$m\in\boldsymbol M$.  Consequently, the specialization of the
operator $\pi_{\boldsymbol A,\boldsymbol M}(x)$ is
\[
\pi_{\boldsymbol A|_0,\boldsymbol M|_0}(1\otimes x).
\]

The constructions above and the results of Section
\ref{thedeformationfamily} give $\widetilde K$-equivariant
isomorphisms
\[
\Phi_A:\boldsymbol A_d|_0\overset{\simeq}{\longrightarrow}A_0,
\qquad
\Phi_S:\boldsymbol S_d|_0\overset{\simeq}{\longrightarrow}S_0.
\]
Here $\Phi_A$ is precisely the algebra isomorphism already
constructed in Section \ref{thedeformationfamily}.  Its construction
combines the zero-fiber isomorphism for universal enveloping algebras
from Subsubsection \ref{414}, the zero-fiber isomorphism for Clifford
algebras from Subsubsection \ref{423}, and the tensor-product
identification \eqref{eq:tensor-base-change}.   

Here $\Phi_S$ is the inverse of the zero fiber of $T_S$, followed by the isomorphism
$(R\otimes_{\C}S_0)|_0\simeq S_0$.  By the last lemma of Section
\ref{thedeformationfamily},
\[
\Phi_A\bigl(1\otimes D(\g_d,\boldsymbol\beta_d)\bigr)
=D_0.
\]
Under $\chi_{\boldsymbol A_d}$, the element
$1\otimes D(\g_d,\boldsymbol\beta_d)$ is the quotient class
$D(\g_d,\boldsymbol\beta_d)+I_0\boldsymbol A_d$ used in that lemma.

Put $V_0:=\boldsymbol V|_0$ and equip it with the
$(\fg_0,K)$-module structure transported through
$\eta:\g_d|_0\longrightarrow\fg_0$.  There is a natural
$\widetilde K$-equivariant isomorphism
\[
\Phi_{\boldsymbol V}:
\boldsymbol M(\boldsymbol V)|_0
\overset{\simeq}{\longrightarrow}V_0\otimes_{\C}S_0,
\]
given  explicitly by 
\[
\Phi_{\boldsymbol V}
:=(\mathbb I_{V_0}\otimes\Phi_S)
\circ\vartheta_{\boldsymbol V,\boldsymbol S_d}^{-1}.
\]
 
\begin{lemma*}[Compatibility of the Dirac operator with the zero fiber]
Under the canonical isomorphism
\[
\Phi_{\boldsymbol V}:
\boldsymbol M(\boldsymbol V)|_0
\overset{\simeq}{\longrightarrow}
V_0\otimes_{\C}S_0,
\]
the specialization of $\delta_{\boldsymbol V}$ is identified with the
Dirac operator $D_0(V_0)$.  More precisely,
\begin{equation}
\Phi_{\boldsymbol V}\circ(\delta_{\boldsymbol V}|_0)
=
D_0(V_0)\circ\Phi_{\boldsymbol V}.
\label{eq:fiber-intertwining}
\end{equation}
Equivalently, the following diagram commutes:
\[
\xymatrix@C=55pt{
\boldsymbol M(\boldsymbol V)|_0
\ar[r]^{\delta_{\boldsymbol V}|_0}
\ar[d]_{\Phi_{\boldsymbol V}}^{\simeq}
&
\boldsymbol M(\boldsymbol V)|_0
\ar[d]^{\Phi_{\boldsymbol V}}_{\simeq}
\\
V_0\otimes_{\C}S_0
\ar[r]_{D_0(V_0)}
&
V_0\otimes_{\C}S_0.
}
\]
\end{lemma*}

\begin{proof}
Put
\[
\boldsymbol D:=D(\g_d,\boldsymbol\beta_d),
\qquad
\overline{\boldsymbol D}:=1\otimes\boldsymbol D
\in\boldsymbol A_d|_0.
\]
For $a\in\C_0$ and
$m\in\boldsymbol M(\boldsymbol V)$, formula
\eqref{eq:specialized-action} gives
\begin{align*}
(\delta_{\boldsymbol V}|_0)(a\otimes m)
&=
a\otimes\delta_{\boldsymbol V}(m)\\
&=
a\otimes
\pi_{\boldsymbol A_d,\boldsymbol V}
(\boldsymbol D)(m)\\
&=
\widetilde{\pi}_{\boldsymbol A_d|_0,V_0}
\bigl(
\overline{\boldsymbol D}\otimes(a\otimes m)
\bigr).
\end{align*}
Hence 
\[
\delta_{\boldsymbol V}|_0
=
\pi_{\boldsymbol A_d|_0,V_0}
(\overline{\boldsymbol D}).
\]

By the definitions of $\Phi_A$ and $\Phi_{\boldsymbol V}$, the
corresponding action maps satisfy
\[
\Phi_{\boldsymbol V}
\bigl(
\widetilde{\pi}_{\boldsymbol A_d|_0,V_0}
(x_0\otimes m_0)
\bigr)
=
\widetilde{\pi}_{A_0,V_0}
\bigl(
\Phi_A(x_0)\otimes\Phi_{\boldsymbol V}(m_0)
\bigr)
\]
for every $x_0\in\boldsymbol A_d|_0$ and
$m_0\in\boldsymbol M(\boldsymbol V)|_0$.
This follows directly, on
pure tensors, from the zero-fiber identifications for the universal
enveloping algebra and the Clifford algebra, together with the
definition of $\Phi_S$ and the tensor-product identification
\eqref{eq:tensor-base-change}.

Since
\[
\Phi_A(\overline{\boldsymbol D})=D_0,
\]
we obtain, for every
$m_0\in\boldsymbol M(\boldsymbol V)|_0$,
\begin{align*}
\Phi_{\boldsymbol V}
\bigl((\delta_{\boldsymbol V}|_0)(m_0)\bigr)
&=
\Phi_{\boldsymbol V}
\bigl(
\widetilde{\pi}_{\boldsymbol A_d|_0,V_0}
(\overline{\boldsymbol D}\otimes m_0)
\bigr)\\
&=
\widetilde{\pi}_{A_0,V_0}
\bigl(
\Phi_A(\overline{\boldsymbol D})
\otimes\Phi_{\boldsymbol V}(m_0)
\bigr)\\
&=
\widetilde{\pi}_{A_0,V_0}
\bigl(
D_0\otimes\Phi_{\boldsymbol V}(m_0)
\bigr)\\
&=
\pi_{A_0,V_0}(D_0)
\bigl(\Phi_{\boldsymbol V}(m_0)\bigr)\\
&=
D_0(V_0)\bigl(\Phi_{\boldsymbol V}(m_0)\bigr).
\end{align*}
This proves \eqref{eq:fiber-intertwining}.
\end{proof}
\begin{proposition*}
The isomorphism $\Phi_{\boldsymbol V}$ induces a canonical
$\widetilde K$-equivariant isomorphism
\[
\overline\Phi_{\boldsymbol V}:
\frac{\ker(\delta_{\boldsymbol V}|_0)}
{\ker(\delta_{\boldsymbol V}|_0)
 \cap\operatorname{im}(\delta_{\boldsymbol V}|_0)}
\overset{\simeq}{\longrightarrow}
H_{D_0}(V_0),
\]
sending the class of $m_0\in\ker(\delta_{\boldsymbol V}|_0)$ to the
class of $\Phi_{\boldsymbol V}(m_0)$.
\end{proposition*}

\begin{proof}
Equation \eqref{eq:fiber-intertwining} says that
the isomorphism $\Phi_{\boldsymbol V}$ intertwines the two linear operators
$\delta_{\boldsymbol V}|_0$ and $D_0(V_0)$.  It therefore
identifies their kernels, their images, and the intersections of
their kernels and images.  Passing to the corresponding quotients
proves the proposition.
\end{proof}

\subsection{The canonical comparison morphism}\label{s53}
In this subsection we construct a canonical $\widetilde K$-equivariant
morphism
\[
\mu_{\boldsymbol V}:
H_{D(\g_d,\boldsymbol\beta_d)}(\boldsymbol V)|_0
\longrightarrow
H_{D_0}(V_0),
\]
which we call the \textit{comparison morphism} of $\boldsymbol V$.

Recall that
\[
q_0=q_{0,\boldsymbol M(\boldsymbol V)}:
\boldsymbol M(\boldsymbol V)
\longrightarrow
\boldsymbol M(\boldsymbol V)|_0,
\qquad
m\longmapsto1\otimes m,
\]
is the natural map to the zero fiber. Under the identification
$\chi_{\boldsymbol M(\boldsymbol V)}$, it is the quotient map modulo
$z$. Since
$\delta_{\boldsymbol V}|_0=\mathbb I_{\C_0}\otimes\delta_{\boldsymbol V}$,
we have
\[
q_0\bigl(\delta_{\boldsymbol V}(m)\bigr)
=
(\delta_{\boldsymbol V}|_0)\bigl(q_0(m)\bigr),
\qquad
m\in\boldsymbol M(\boldsymbol V),
\]
and hence
\[
q_0(\ker\delta_{\boldsymbol V})
\subseteq
\ker(\delta_{\boldsymbol V}|_0).
\]
Moreover, if
\[
s\in
\ker\delta_{\boldsymbol V}
\cap
\operatorname{im}\delta_{\boldsymbol V},
\]
then $s=\delta_{\boldsymbol V}(m)$ for some
$m\in\boldsymbol M(\boldsymbol V)$, and hence
\[
q_0(s)
=
(\delta_{\boldsymbol V}|_0)\bigl(q_0(m)\bigr)
\in\operatorname{im}(\delta_{\boldsymbol V}|_0).
\]
It follows that
\[
q_0\bigl(
\ker\delta_{\boldsymbol V}
\cap
\operatorname{im}\delta_{\boldsymbol V}
\bigr)
\subseteq
\ker(\delta_{\boldsymbol V}|_0)
\cap
\operatorname{im}(\delta_{\boldsymbol V}|_0).
\]
Consequently, $q_0$ induces a canonical $\widetilde K$-equivariant
morphism
\[
\overline q_0:
H_{D(\g_d,\boldsymbol\beta_d)}(\boldsymbol V)
\longrightarrow
\frac{\ker(\delta_{\boldsymbol V}|_0)}
{\ker(\delta_{\boldsymbol V}|_0)
 \cap\operatorname{im}(\delta_{\boldsymbol V}|_0)},
\]
given by
\[
\overline q_0
\bigl(s+\ker\delta_{\boldsymbol V}\cap\operatorname{im}\delta_{\boldsymbol V}\bigr)
=
q_0(s)+\ker(\delta_{\boldsymbol V}|_0)\cap\operatorname{im}(\delta_{\boldsymbol V}|_0),
\qquad
s\in\ker\delta_{\boldsymbol V}.
\]
Let
\[
\overline\Phi_{\boldsymbol V}:
\frac{\ker(\delta_{\boldsymbol V}|_0)}
{\ker(\delta_{\boldsymbol V}|_0)
 \cap\operatorname{im}(\delta_{\boldsymbol V}|_0)}
\overset{\simeq}{\longrightarrow}
H_{D_0}(V_0)
\]
be the $\widetilde K$-equivariant isomorphism of the proposition in
Subsection \ref{s52}. The composite
\[
\overline\Phi_{\boldsymbol V}\circ\overline q_0:
H_{D(\g_d,\boldsymbol\beta_d)}(\boldsymbol V)
\longrightarrow
H_{D_0}(V_0)
\]
is $\widetilde K$-equivariant, and it is $R$-linear when the target
is regarded as an $R$-module through the evaluation homomorphism
\[
R\longrightarrow\C_0\simeq\C.
\]
In particular, $I_0$ acts trivially on the target. The composite
therefore vanishes on
\[
I_0H_{D(\g_d,\boldsymbol\beta_d)}(\boldsymbol V)
\]
and factors uniquely through the zero fiber
$H_{D(\g_d,\boldsymbol\beta_d)}(\boldsymbol V)|_0$ of the Dirac
cohomology of $\boldsymbol V$. We have proved the following.

\begin{corollary*}
There is a canonical $\widetilde K$-equivariant morphism
\begin{equation}
\mu_{\boldsymbol V}:
H_{D(\g_d,\boldsymbol\beta_d)}(\boldsymbol V)|_0
\longrightarrow
H_{D_0}(V_0).
\label{eq:canonical-comparison}
\end{equation}
More explicitly, if $s\in\ker\delta_{\boldsymbol V}$, then
\[
\mu_{\boldsymbol V}
\left(
1\otimes
\left(
s+
\ker\delta_{\boldsymbol V}
\cap\operatorname{im}\delta_{\boldsymbol V}
\right)
\right)
=
\Phi_{\boldsymbol V}(q_0(s))
+
\ker D_0(V_0)\cap\operatorname{im}D_0(V_0).
\]
\end{corollary*}
\subsection{An injectivity criterion for the comparison morphism}\label{s54}
In this subsection we give a sufficient condition for the comparison
morphism $\mu_{\boldsymbol V}$ of \eqref{eq:canonical-comparison} to be
injective.


\begin{proposition*}
Let $\boldsymbol V$ be an algebraic family of Harish-Chandra modules
for $(\g_d,K)$, and put $V_0:=\boldsymbol V|_0$.  Suppose that
\[
\ker D_0(V_0)\cap\operatorname{im}D_0(V_0)=0.
\]
Then the comparison morphism $\mu_{\boldsymbol V}$ of
\eqref{eq:canonical-comparison} is injective.
\end{proposition*}

\begin{proof}
Put
\[
\boldsymbol K_{\boldsymbol V}
:=\ker\delta_{\boldsymbol V},
\qquad
\boldsymbol J_{\boldsymbol V}
:=
\ker\delta_{\boldsymbol V}
\cap\operatorname{im}\delta_{\boldsymbol V}.
\]
Thus
\[
H_{D(\g_d,\boldsymbol\beta_d)}(\boldsymbol V)
=
\boldsymbol K_{\boldsymbol V}/\boldsymbol J_{\boldsymbol V}.
\]

By Subsection \ref{s53},
\[
q_0(\boldsymbol J_{\boldsymbol V})
\subseteq
\ker(\delta_{\boldsymbol V}|_0)
\cap\operatorname{im}(\delta_{\boldsymbol V}|_0).
\]
Under the isomorphism $\Phi_{\boldsymbol V}$, the space on the
right is identified, by \eqref{eq:fiber-intertwining}, with
\[
\ker D_0(V_0)\cap\operatorname{im}D_0(V_0),
\]
which is zero by hypothesis.  Hence
\[
q_0(\boldsymbol J_{\boldsymbol V})=0.
\]

Let $s\in\boldsymbol J_{\boldsymbol V}$.  Since $q_0(s)=0$, the
quotient description of the zero fiber as $\boldsymbol M(\boldsymbol V)|_0=\boldsymbol M(\boldsymbol V)/I_0\boldsymbol M(\boldsymbol V)$ gives
\[
s=zm
\]
for some $m\in\boldsymbol M(\boldsymbol V)$.  Moreover,
\[
0=\delta_{\boldsymbol V}(s)
=z\delta_{\boldsymbol V}(m).
\]

We now explain why multiplication by $z$ on $\boldsymbol M(\boldsymbol V)$ is injective. 
Since $\boldsymbol V$ is flat over $R$ and $\boldsymbol S_d$ is free
of finite rank over $R$, their tensor product, $\boldsymbol M(\boldsymbol V)$, is
flat over $R$.  As a flat module over a domain, $\boldsymbol M(\boldsymbol V)$ is torsion-free, see
\cite[Prop.~3.49]{rotmanHomologicalAlgebra}; in particular, multiplication by $z$ on
$\boldsymbol M(\boldsymbol V)$ is injective.

The last equation and injectivity of multiplication 
 by $z$ on $\boldsymbol M(\boldsymbol V)$ imply that 
 $\delta_{\boldsymbol V}(m)=0$.  Thus
$m\in\boldsymbol K_{\boldsymbol V}$, and consequently
\[
\boldsymbol J_{\boldsymbol V}
\subseteq z\boldsymbol K_{\boldsymbol V}=I_0 \boldsymbol K_{\boldsymbol V}.
\]
It follows that
\begin{align*}
H_{D(\g_d,\boldsymbol\beta_d)}(\boldsymbol V)|_0
&=
(\boldsymbol K_{\boldsymbol V}/
 \boldsymbol J_{\boldsymbol V})|_0=(\boldsymbol K_{\boldsymbol V}/
 \boldsymbol J_{\boldsymbol V})/(I_0(\boldsymbol K_{\boldsymbol V}/
 \boldsymbol J_{\boldsymbol V}))\\
 & =(\boldsymbol K_{\boldsymbol V}/
 \boldsymbol J_{\boldsymbol V})/((I_0\boldsymbol K_{\boldsymbol V}+ \boldsymbol J_{\boldsymbol V})/
 \boldsymbol J_{\boldsymbol V})\\
&\simeq
\frac{\boldsymbol K_{\boldsymbol V}}
{\boldsymbol J_{\boldsymbol V}
 +I_0\boldsymbol K_{\boldsymbol V}}=
\frac{\boldsymbol K_{\boldsymbol V}}
{I_0\boldsymbol K_{\boldsymbol V}}
=
\boldsymbol K_{\boldsymbol V}|_0.
\end{align*}

On the other hand, the hypothesis on the zero-fiber operator gives
\[
H_{D_0}(V_0)=\ker D_0(V_0).
\]
Under these identifications, by the corollary of Subsection \ref{s53},
$\mu_{\boldsymbol V}$ is the morphism
\[
\boldsymbol K_{\boldsymbol V}|_0
\longrightarrow
\ker D_0(V_0)
\]
that sends $1\otimes s$ to
$\Phi_{\boldsymbol V}(q_0(s))$.

It remains to prove that this morphism is injective.  Suppose that
$s\in\boldsymbol K_{\boldsymbol V}$ and
\[
\Phi_{\boldsymbol V}(q_0(s))=0.
\]
Since $\Phi_{\boldsymbol V}$ is an isomorphism, $q_0(s)=0$.  Hence
$s=zm$ for some $m\in\boldsymbol M(\boldsymbol V)$.  As above,
\[
0=\delta_{\boldsymbol V}(s)
=z\delta_{\boldsymbol V}(m)
\]
implies that $m\in\boldsymbol K_{\boldsymbol V}$.  Therefore
$s\in z\boldsymbol K_{\boldsymbol V}$, and so
$1\otimes s=0$ already in
$\boldsymbol K_{\boldsymbol V}|_0$.  This proves that
$\mu_{\boldsymbol V}$ is injective.\end{proof}

\begin{corollary*}
Let $\boldsymbol V$ be an algebraic family of Harish-Chandra modules
for $(\g_d,K)$ such that $V_0:=\boldsymbol V|_0$ is
infinitesimally unitary.
Then the comparison morphism $\mu_{\boldsymbol V}$ of
\eqref{eq:canonical-comparison} is injective.
\end{corollary*}

\begin{proof}
Choose a
$\beta_0$-orthonormal basis $\{e_1,\ldots,e_m\}$ of
$\fp_0^\sigma$, and equip $S_0$ with a positive-definite Hermitian
form for which the real Clifford generators are self-adjoint. Then
\[
D_0(V_0)
=
\sum_{i=1}^m
\pi_{\mathcal U(\fg_0),V_0}(e_i)
\otimes
\gamma'_{\fp_0^-,\beta_0}\bigl(\gamma_0(e_i)\bigr).
\]
Infinitesimal unitarity gives
\[
\pi_{\mathcal U(\fg_0),V_0}(e_i)^*
=
-\pi_{\mathcal U(\fg_0),V_0}(e_i),
\]
whereas
\[
\gamma'_{\fp_0^-,\beta_0}\bigl(\gamma_0(e_i)\bigr)^*
=
\gamma'_{\fp_0^-,\beta_0}\bigl(\gamma_0(e_i)\bigr).
\]
Consequently,
\[
D_0(V_0)^*=-D_0(V_0).
\]
This is the usual unitary adjointness argument for algebraic Dirac
operators, adapted to our Clifford convention; compare
\cite[Remark~3.2.4]{pandzic}.

If
\[
y\in
\ker D_0(V_0)\cap\operatorname{im}D_0(V_0),
\]
write $y=D_0(V_0)x$. Then
\[
\langle y,y\rangle
=
\langle D_0(V_0)x,y\rangle
=
-\langle x,D_0(V_0)y\rangle
=
0.
\]
Hence $y=0$. Thus infinitesimal unitarity implies the required
kernel--image condition.  
\end{proof}

\begin{example*}
 Let $V$ be a nonzero
infinitesimally unitary $(\fg,K)$-module for which the usual Dirac
cohomology $H_D(V)$ vanishes, and consider the restricted constant
family
\[
\boldsymbol V:=R\otimes_{\C}V
\]
as a $(\g_d,K)$-module.  The compatibility formula for the family
Dirac operator in
\cite[Sec.~4.7]{afentoulidisalmpanis2025diracoperatorsalgebraicfamilies}
shows that
\[
\delta_{\boldsymbol V}=z(1_R\otimes D(V)).
\]

By the unitary adjointness result
\cite[Remark~3.2.4]{pandzic}, we have
$H_D(V)=\ker D(V)$, so $D(V)$ is injective on
$V\otimes_{\C}S_0$.  Hence
\[
1_R\otimes D(V):
R\otimes_{\C}(V\otimes_{\C}S_0)
\longrightarrow
R\otimes_{\C}(V\otimes_{\C}S_0)
\]
is injective.  Under the canonical identification
\[
\boldsymbol M(\boldsymbol V)
\simeq
R\otimes_{\C}(V\otimes_{\C}S_0),
\]
multiplication by $z$ is also injective.  Since
\[
\delta_{\boldsymbol V}
=
z(1_R\otimes D(V)),
\]
it follows that $\ker\delta_{\boldsymbol V}=0$.

   In particular, the family has
trivial Dirac kernel--image intersection.  On the zero fiber,
$\fp_0$ acts trivially, and hence $D_0(V_0)=0$.  Consequently,
\[
H_{D(\g_d,\boldsymbol\beta_d)}(\boldsymbol V)|_0=0,
\qquad
H_{D_0}(V_0)=V_0\otimes_{\C}S_0.
\]
Thus $\mu_{\boldsymbol V}$ is injective but not surjective.  This example explicitly exhibits the failure of Dirac
cohomology to commute with specialization at zero.
\end{example*}
\printbibliography
\Addresses
\end{document}